\documentclass[12pt]{amsart}

\usepackage{amsthm, amsmath, amssymb, url, verbatim, amscd}
\usepackage[margin=1in]{geometry}

\usepackage{hyperref}
\hypersetup{
colorlinks,
allcolors=blue
}

\usepackage{enumerate}

\usepackage{latexsym}
\usepackage{amssymb}
\usepackage{amsmath}
\usepackage{amsfonts}
\usepackage[mathscr]{eucal}
\usepackage{color}
\usepackage{tabulary}
\usepackage{stmaryrd}

\usepackage[all]{xy}
\usepackage{amscd}

\usepackage{graphicx}
\usepackage{epsfig}
\usepackage{relsize}
\usepackage{tikz}
\usetikzlibrary{matrix}

\usepackage{longtable}
\usepackage{rotating}
\usepackage{array}
\usepackage{multicol}
\usepackage{multirow}
\usepackage{makecell}
\usepackage{mathtools}
  \usepackage{cellspace}
\usepackage{booktabs}

\usepackage{multicol}
\usepackage{color}

\usepackage{lineno}

\usepackage[normalem]{ulem}

\theoremstyle{plain}
\newtheorem{theorem}{Theorem}[section]

\newtheorem{lemma}[theorem]{Lemma}
\newtheorem{proposition}[theorem]{Proposition}

\newtheorem{main}{Theorem}
\newtheorem{Corollary}[main]{Corollary}

\newtheorem{thm}[theorem]{Theorem}
\newtheorem{cor}[theorem]{Corollary}
\newtheorem{lem}[theorem]{Lemma}
\newtheorem{prop}[theorem]{Proposition}

\theoremstyle{definition}

\theoremstyle{remark}
\newtheorem{remark}[theorem]{Remark}

\newtheorem*{ack}{Acknowledgements}
\numberwithin{equation}{section}

\hypersetup{
colorlinks,
allcolors=blue
}

\newcommand{\N}{\mathbb{N}}
\newcommand{\Z}{\mathbb{Z}}
\newcommand{\Q}{\mathbb{Q}}
\newcommand{\R}{\mathbb{R}}
\newcommand{\C}{\mathbb{C}}

\DeclareMathOperator{\SO}{SO}
\DeclareMathOperator{\Or}{O}
\DeclareMathOperator{\SU}{SU}

\DeclareMathOperator{\Spin}{Spin}
\DeclareMathOperator{\Sp}{Sp}
\newcommand{\Gtwo}{\operatorname G_2}
\newcommand{\Ffour}{\operatorname F_4}

\newcommand{\Eseven}{\operatorname E_7}
\newcommand{\Eeight}{\operatorname E_8}
\newcommand{\bq}{/\!\!/}
\newcommand{\CP}{\mathbf{C}P}
\newcommand{\ox}{\otimes}
\newcommand{\rk}{\operatorname{rank}}

\DeclareMathOperator{\codim}{codim}

\DeclareMathOperator{\cat}{cat}

\DeclareMathOperator{\rank}{rank}

\DeclareMathOperator{\im}{Im}

\DeclareMathOperator{\Diff}{Diff}

\newcommand{\sph}{\mathbf{S}}

\newcommand{\RP}{\mathbf{R}P}

\renewcommand{\leq}{\leqslant}
\renewcommand{\geq}{\geqslant}
\newcommand{\rnorm}{\trianglerighteqslant}

\def\ol{\overline}
\def\ul{\underline}
\def\wt{\widetilde}

\def\x{\times}
\def\ox{\otimes}

\def\lra{\longrightarrow}

\def\In{\subseteq}

\def\bs{\backslash}

\def\<{\langle}
\def\>{\rangle}

\def\ve{\varepsilon}
\def\vphi{\varphi}

\def\bpm{\begin{pmatrix}}
\def\epm{\end{pmatrix}}
\def\bvm{\begin{vmatrix}}
\def\evm{\end{vmatrix}}
\def\bsm{\left(\begin{smallmatrix}}
\def\esm{\end{smallmatrix}\right)}

\def\beq{\begin{equation}}
\def\eeq{\end{equation}}

\begin{document}

 \title{Manifolds that admit a double disk-bundle decomposition}

\author[J.\ DeVito]{Jason DeVito}
\address[J.\ DeVito]{Department of Mathematics, University of Tennessee at Martin, USA.}
\email{jdevito1@utm.edu}

\author[F.~Galaz-Garc\'{i}a]{Fernando\ Galaz-Garc\'{i}a}
\address[F.~Galaz-Garc\'{i}a]{Department of Mathematical Sciences, Durham University, United Kingdom.	
					    }
\email{fernando.galaz-garcia@durham.ac.uk}


\author[M.~Kerin]{Martin Kerin}
\address[M.~Kerin]{School of Mathematical \& Statistical Sciences, NUI
	Galway, Ireland.}
\email{martin.kerin@nuigalway.ie}



\begin{abstract}
Under mild topological restrictions, this article establishes that a smooth, closed, simply connected manifold of dimension at most seven which can be decomposed as the union of two disk bundles must be rationally elliptic.  In dimension five, such manifolds are classified up to diffeomorphism, while the same is true in dimension six when either the second Betti number vanishes or the third Betti number is non-trivial.
\end{abstract}

\date{\today}

\maketitle


\section{Introduction}
A closed manifold is said to admit a \emph{double disk-bundle decomposition} if it can be written as the union of two disk bundles glued together along their common boundary by a diffeo\-morphism.  For example, a sphere $\sph^n$, $n \geq 2$, is well known to admit at least two such decompositions: $D^n \cup D^n$ and $(\sph^p \x D^{q+1}) \cup (D^{p+1} \x \sph^q)$, where $n = p + q + 1$.

Frequently, in the differential geometry literature, such decompositions either arise naturally from geometric hypotheses (see, for example, \cite{Heb}, \cite{Mue} and \cite{ScSe}) or are used to create novel, often non-homogeneous, examples of certain interesting phenomena.  Although it would be impossible to give an exhaustive listing, it is perhaps instructive to highlight just some of the many situations where double disk-bundle decompositions appear.  

In the study of isoparametric and Dupin hypersurfaces, double disk-bundle decompositions play a central role; see, for example, \cite{Mue}, \cite{GH} and \cite{St}.  In \cite{TZ}, it was shown that all fake quaternionic projective planes (see \cite{EeKu}) admit a Riemannian metric such that there is a point through which all geodesics are simply closed and of the same length.  As particular examples of singular Riemannian foliations \cite{QT}, double disk-bundle decompositions are also well understood from the point of view of mean curvature flow \cite{AR}.

There is a vast literature dealing with the special case of cohomogeneity-one manifolds, where the decomposition arises as a result of the existence of a Lie group action with one-dimensional orbit space.  For example, the additional symmetry afforded by such an action has been exploited to study minimal hypersurfaces in spheres \cite{HL}, construct infinite families of inhomogeneous Einstein manifolds \cite{Boe} and construct new examples of inhomogeneous nearly-K\"ahler structures on $6$-manifolds \cite{FoHa}.

In the study of positive and non-negative sectional curvature, the presence of a decomposition as the union of two disk bundles has proven useful both in producing exciting new examples (see, for example, \cite{De}, \cite{GKS1}, \cite{GKS2}, \cite{GVZ} and \cite{GZ}) and in proving classification results under additional symmetry assumptions (see, for example, \cite{GS1}, \cite{GS2} and \cite{GVWZ}).  In particular, this extra structure has led to the proof of some special cases (see \cite{GKRW} and \cite{Sp}) of the Bott Conjecture, which asserts that a closed, simply connected Riemannian manifold admitting a metric with non-negative sectional curvature must be rationally elliptic.  Recall that a closed manifold $M$ is said to be rationally elliptic if $\dim_\Q(\pi_*(M) \ox \Q) < \infty$, and rationally hyperbolic otherwise.

Given the Bott Conjecture, the prevalence of double disk-bundle decompositions among known examples of manifolds admitting non-negative or positive sectional curvature, and given that the Double-Soul Conjecture \cite{Gr} asks whether every non-negatively curved, closed, simply connected Riemannian manifold admits a double disk-bundle decomposition, the present work is motivated by a desire to understand whether there is any connection between rational ellipticity and these decompositions, even independent of curvature assumptions.

After a moment's thought, it is clear that some topological restrictions are necessary in any such investigation.  Indeed, for all $n \geq 2$ and all $m \geq 3$, the $(n + 4)$-dimensional manifold $\sph^n \x \#_{k = 1}^m \CP^2$ is rationally hyperbolic and admits a double disk-bundle decomposition.  Nevertheless, it turns out that, in low dimensions, the required topological restrictions are very mild.

\begin{main}
\label{T:thmA}
Let $M^n$ be a smooth, closed, simply connected manifold of dimension $n \leq 7$ which admits a double disk-bundle decomposition.  Then $M^n$ is rationally elliptic if and only if either $n \leq 5$ or else $n = 6$ and $b_2(M^6) \leq 3$ (respectively, $n = 7$ and $b_2(M^7) \leq 2$).
\end{main}

Note that, if $M^6$ (respectively, $M^7$) is rationally elliptic, then it is well known that $b_2(M^6) \leq 3$ (respectively, $b_2(M^7) \leq 2$); see, for example, \cite{He}.  Therefore, in light of the examples given above, the statement is optimal in respect of restrictions on the second Betti number in dimensions $\leq 7$.  Furthermore, notice that Theorem \ref{T:thmA} may be restated as follows: a smooth, closed, simply connected manifold of dimension $\leq 7$ which admits a double-disk bundle decomposition is rationally elliptic if and only if it has the same Betti numbers as a rationally elliptic manifold.  This statement is optimal in respect of dimension since there are counterexamples already in dimension eight.  For example, for all $n \geq 2$, the $(n+6)$-dimensional manifold $\sph^n \x ((\sph^2 \x \sph^4) \# (\sph^2 \x \sph^4))$ is rationally hyperbolic, yet admits a double disk-bundle decomposition and has the same Betti numbers as the rationally elliptic space $\sph^n \x \sph^2 \x \CP^2$.  For completeness, however, recall that a consequence of the work of Miller \cite{Mill} is that, for $k \geq 2$, a smooth, closed, $(k-1)$-connected manifold of dimension $\leq 4k-2$ is rationally elliptic whenever its rational cohomology ring is isomorphic to that of a rationally elliptic space. 

Closed, smooth, simply connected manifolds of dimension four which admit a double disk-bundle decomposition were classified up to diffeomorphism in \cite{GeRa}.  The only such manifolds are $\sph^4$, $\CP^2$, $\sph^2 \x \sph^2$ and $\CP^2 \# \pm \CP^2$, precisely those simply connected $4$-manifolds known to admit a metric of non-negative sectional curvature.  In dimension five, it turns out that an analogous statement is true.


\begin{main}
\label{T:DIFFEO_CLASSIF}
A smooth, closed, simply connected manifold of dimension five admits a double disk-bundle decomposition if and only if it is diffeomorphic to $\sph^5$, the Wu manifold $\SU(3)/\SO(3)$, $\sph^3 \x \sph^2$, or 
the unique non-trivial 
$\sph^3$-bundle over $\sph^2$.
\end{main}

Theorem \ref{T:DIFFEO_CLASSIF} may be viewed as further evidence that this is the complete list of simply connected $5$-manifolds admitting a metric of non-negative sectional curvature; for example, see \cite{GGS}, where decompositions as the union of two disk bundles play a key role.  Note that, by \cite{Ho}, each of the manifolds in Theorem \ref{T:DIFFEO_CLASSIF} admits a cohomogeneity-one action and, hence, a double disk-bundle decomposition.  

Recall that a closed, simply connected, rationally elliptic $5$-manifold must be rationally homotopy equivalent to either $\sph^5$ or $\sph^ 3 \x \sph^2$; see, for example, \cite{Pav}.  However, each of these rational homotopy types contains infinitely many distinct homotopy types.  Indeed, since the Wu manifold $\SU(3)/\SO(3)$ is a rational homology $5$-sphere with $H_2(\SU(3)/\SO(3); \Z) = \Z_2$, taking the connected sum of a smooth, closed, simply connected $5$-manifold with $\SU(3)/\SO(3)$ will change its homotopy type, but not its rational homotopy type.


\begin{Corollary}
There are infinitely many smooth, closed, simply connected, rationally elliptic $5$-manifolds which do not admit a double disk-bundle decomposition.
\end{Corollary}

It is an intriguing coincidence that the diffeomorphism types of manifolds of dimension $\leq 5$ which admit a double disk-bundle decomposition are precisely those for which there exists a Riemannian metric with trivial topological entropy; see \cite{PP}.  It would be interesting to know whether this is also true in higher dimensions.

In dimensions six and seven, it is well known that there are infinitely many rational homotopy types of closed, smooth, simply connected manifolds; see \cite{He}, \cite{To}.  Many of these rational homotopy types can be represented by a nice model space which, by Proposition \ref{P:SuffCond}, can easily be seen to admit a double disk-bundle decomposition.  In particular, an infinite family of $2$-connected, rational homology $7$-spheres, each admitting infinitely many double disk-bundle decompositions, was constructed in \cite{GKS1}, including many spaces which are not even homotopy equivalent to an $\sph^3$-bundle over $\sph^4$ \cite{GKS2}.  On the other hand, there are infinitely many rational homotopy types for which no nice representative is known, nor whether any representative can be decomposed as the union of two disk bundles.  Therefore, a classification in dimensions six and seven up to diffeomorphism, similar to that in Theorem \ref{T:DIFFEO_CLASSIF}, seems beyond the scope of the present article.  Nevertheless, it is possible to obtain some partial results.

\begin{main}
\label{T:6dim_b2}
A smooth, closed, simply connected $6$-manifold $M^6$ with $b_2(M^6) = 0$ admits a double disk-bundle decomposition if and only if it is diffeomorphic to either $\sph^6$ or $\sph^3 \x \sph^3$.
\end{main} 

Observe that the conclusion of Theorem \ref{T:6dim_b2} excludes all smooth, closed, simply connected $6$-manifolds $M^6$ with $b_2(M^6) = 0$ which have torsion in their cohomology, including all (non-trivial) rational homology spheres.

\begin{Corollary}
There are infinitely many smooth, closed, simply connected, rationally elliptic $6$-manifolds which do not admit a double disk-bundle decomposition.
\end{Corollary}

In fact, even without any assumption on $b_2(M^6)$ in Theorem \ref{T:6dim_b2} above, $M^6$ must be diffeomorphic to $\sph^3 \x \sph^3$ if $b_3(M^6) \neq 0$.  Therefore, in combination with Theorem \ref{T:thmA}, it follows that a rationally hyperbolic $6$-manifold which admits a double disk-bundle decomposition has its rational cohomology concentrated in even degrees and Euler characteristic $\chi \geq 10$.

\begin{Corollary}
If $M^6$ is a smooth, closed, simply connected, rationally hyperbolic $6$-manifold which admits a double disk-bundle decomposition, then $b_2(M^6) \geq 4$ and $b_3(M^6) = 0$.
\end{Corollary}

The results in this article may be viewed as evidence that admitting a double disk-bundle decomposition imposes strong restrictions on the topology of a manifold.  Consequently, it might be hoped that, in general, the rational homotopy type of such a manifold is 
determined by its rational cohomology ring, a property known as formality.  It follows from work of Miller \cite{Mill} that all closed, simply connected manifolds of dimension $\leq 6$ are (intrinsically) formal, while, by recent work of Crowley and Nordstr\"om \cite{CN}, a closed, simply connected $7$-manifold is (intrinsically) formal if its cohomology ring satisfies a certain hard Lefschetz property.

\begin{main}
\label{T:nonformal}
There are infinitely many non-formal, smooth, closed, simply connected, rationally elliptic $7$-manifolds which admit a double disk-bundle decomposition.
\end{main}

The manifolds in Theorem \ref{T:nonformal} are a certain family of biquotients of the form $(\sph^3 \x \sph^3 \x \sph^3) \bq T^2$, all of the same rational homotopy type and distinguished by the order of the torsion in their cohomology rings.  In particular, the unit tangent bundle of $\sph^2 \x \sph^2$ is one such space.

Returning to the original motivations for this work, we conclude this introduction with some final observations which are likely already well known to the experts.  First, every known example of a simply connected manifold admitting a Riemannian metric of positive sectional curvature admits a double disk-bundle decomposition; see Theorem \ref{T:possec}.  While this is evidence for the validity of the Double-Soul Conjecture, the conjecture is completely open even in some of the simplest cases of manifolds admitting non-negative curvature.  Indeed, among compact Lie groups, it is currently unknown whether a semi-simple Lie group with all simple factors being either $E_7$ or $E_8$ can be decomposed as the union of two disk bundles; see Lemma \ref{L:GroupDDBD}.

Ever since the discovery of exotic spheres, there has been interest in determining to what extent their geometry resembles that of the standard sphere.  In dimension seven, it is now known that all exotic spheres admit a metric of non-negative sectional curvature \cite{GKS1}.  The key to obtaining such a metric is the result of Grove and Ziller ensuring that every cohomogeneity-one manifold with codimension-two singular orbits admits such a metric \cite{GZ}.  It is natural to ask whether something similar will work for higher-dimensional exotic spheres.  As it turns out, a (rational homology) sphere can be decomposed as the union of two $2$-disk bundles only if it has dimension $\leq 7$; see Corollary \ref{C:spheres} and also \cite{Fa}.  Therefore, new techniques and ideas will be required to construct a metric with non-negative curvature on a higher-dimensional exotic sphere.

One of the main tools used to obtain many of the results in the paper, including Corollary \ref{C:spheres}, is the generalization Theorem \ref{T:NoRank} of a result of Grove and Halperin \cite[Lemma 6.3]{GH} showing the non-triviality of the connecting homomorphism in a certain long exact sequence of rational homotopy groups naturally associated to a double disk-bundle decomposition.

\textbf{Organization:} 
In Section \ref{S:Prelims}, the notation to be used throughout the paper is introduced and a summary is provided of the parts of rational homotopy theory relevant to this work.  In Section \ref{S:Examples}, some sufficient conditions are collected which ensure the existence of a double disk-bundle decomposition and then used to examine compact Lie groups and manifolds known to admit positive sectional curvature.  Section \ref{S:top} focuses upon establishing general topological results relevant to manifolds admitting a double disk-bundle decomposition, with the main result being Theorem \ref{T:NoRank}.  
Sections \ref{S:dim5}, \ref{S:DIM_6} and \ref{S:dim7} are devoted to studying double disk-bundle decompositions in dimensions at most five, equal to six and equal to seven, respectively.


\begin{ack} This project was initiated during the conference \emph{Representations of Riemannian Geometry}, held at St.\ Joseph's University in Philadelphia, PA, in August of 2017. The authors thank the conference organizers and the university for their hospitality and a stimulating atmosphere, as well as the National Science Foundation for supporting our attendance at the conference through the grant DMS--1720590.  JD was also individually supported by the NSF via the grant DMS--2105556.  FGG and MK received support from the DFG grants GA 2050/2-1 and KE 2248/1-1 under SPP2026 \emph{Geometry at Infinity}.  FGG received support from the DFG grant 281869850, RTG 2229 \emph{Asymptotic Invariants and Limits of Groups and Spaces}.  MK received support from SFB 878: \emph{Groups, Geometry \& Actions} and the Cluster of Excellence at the Mathematical Institute of the University of M\"unster.  Finally, the authors thank Marco Radeschi for helpful conversations on singular Riemannian foliations.
\end{ack}


\section{Preliminaries}
\label{S:Prelims}


\subsection{Terminology and notation}{\ }

Suppose that $D^{\ell_\pm + 1} \to DB_\pm \to B_\pm$ are smooth disk bundles of rank $\ell_\pm + 1$, respectively, over smooth, closed manifolds $B_\pm$, and that there is a diffeomorphism $f \colon \partial DB_- \to \partial DB_+$ of the boundaries.  Identifying these boundaries via the diffeomorphism $f$, the resulting smooth manifold $M = DB_- \cup_f DB_+$ is called a \emph{double disk bundle}.  If $L$ denotes the common image of $\partial DB_\pm$ in $M$, then it is clear that there are sphere bundles $\sph^{\ell_\pm} \to L \to B_\pm$.  

An arbitrary smooth, closed, connected manifold $M$ is said to admit a \emph{double disk-bundle decomposition} if there exists a diffeomorphism $\Phi \colon M \to DB_- \cup_f DB_+$ from $M$ to a double disk bundle $DB_- \cup_f DB_+$.  By an abuse of notation, $B_\pm$ and $L$ will be used to denote the images of pulling back to $M$ via $\Phi$ the corresponding objects in the double disk bundle $DB_- \cup_f DB_+$, while the decomposition itself will often be denoted by $DB_- \cup_L DB_+$ whenever the precise gluing map is not needed.  As a consequence of Proposition \ref{P:connected} below, $B_\pm$ will usually be assumed to be connected, without additional comment.

It is clear that knowledge of the dimensions $\ell_\pm$ of the fibers of the sphere bundles $\sph^{\ell_\pm} \to L \to B_\pm$ will play a role in understanding the topology of $M$.  Therefore, manifolds admitting a double disk-bundle decomposition will often be discussed under additional restrictions on $\ell_\pm$.  Recall, moreover, that the diffeomorphism group $\Diff(\sph^k)$ deformation retracts onto $\Or(k+1)$ whenever $k \leq  3$ \cite{BK, Ha1, Sm1}.  Hence, it will be implicitly assumed that $\sph^{\ell_\pm} \to L \to B_\pm$ is a linear bundle if $\ell_\pm \leq 3$, respectively.  The inclusion $L \to M$ gives rise to an additional homotopy fibration $F \to L \to M$, where $F$ denotes the so-called \emph{homotopy fiber}.

In \cite{QT}, Qian and Tang showed that every manifold $M$ admitting a double disk-bundle decomposition $DB_- \cup_L DB_+$ also admits a codimension-one singular Riemannian foliation with singular leaves diffeomorphic to $B_\pm$ and regular leaf diffeomorphic to $L$.  Thus, it will at times be convenient to abuse this suggestive terminology and refer to $B_\pm$ and $L$ as the \emph{singular} and \emph{regular leaves},  respectively, of the double disk-bundle decomposition of $M$.

The symbol $\cong$ will be used to indicate either that two manifolds are diffeomorphic or that two groups are isomorphic, depending on the context.  Finally, homology and cohomology will be taken with integral coefficients, unless explicitly indicated otherwise.


\subsection{Rational homotopy theory}{\ }

Borrowing heavily from \cite{GGKR}, the basics of rational homotopy theory required in this work can be summarized as follows.  For a full treatment, see \cite{FHT, FHT2, FOT}.


A path-connected topological space $X$ is said to be \emph{nilpotent} if its fundamental group $\pi_1(X)$ is a nilpotent group which acts nilpotently on the higher homotopy groups $\pi_k(X)$, $k \geq 2$, by the action described in \cite[p.~31]{FOT}.  Recall that a group $G$ acts nilpotently on a group $H$ if there is a finite chain 
$$
H = H_0 \rnorm H_1 \rnorm \cdots \rnorm H_m = \{e\}
$$
of subgroups such that, for each $j \in \{1, \dots, m\}$, $H_{j}$ is normal in $H_{j - 1}$ and closed under the action of $G$, the quotients $H_{j-1}/H_j$ are abelian and the induced action of $G$ on $H_{j-1}/H_j$ is trivial.  In particular, a group $G$ is nilpotent if and only if it acts on itself nilpotently by conjugation.  

{Recall that the \textit{rank of an abelian group} $A$ is the dimension of the rational vector space $A \ox \Q$.}  Building on this, the \emph{rank of a nilpotent group} $G$ is given by
$$
\rank(G) = \sum_{j=1}^{n} \rank \left( G_{j-1}/G_j \right),
$$
where $\{G_j\}_{j = 0}^n$ denotes the lower central series of $G$ and each of the groups $G_{j-1}/G_j$, $j \geq 1$, is abelian; that is, $G_0 = G$ and $G_j = [G_{j-1},G]$ for $j \geq 1$.  In particular, the quaternion group $Q_8 = \{\pm 1, \pm i, \pm j, \pm k\}$ has $\rank(Q_8) = 0$.

Let $X$ be a nilpotent topological space.  The \emph{rational homotopy groups} of $X$ are the $\Q$-vector spaces $\pi_i^\Q(X) =   \pi_i(X) \ox \Q$, $i \geq 2$, of dimension $d_i(X) = \dim_\Q (\pi_i^\Q (X))$.  The space $X$ is \emph{rationally elliptic} if 
\[
\dim_\Q H^*(X; \Q) < \infty \ \text{\ and\ } \ \dim_\Q (\pi_*^\Q (X)) = \sum_{i = 2}^\infty d_i(X) < \infty.
\]
If, instead, $\dim_\Q (\pi_*^\Q (X)) = \infty$, then $X$ is said to be \emph{rationally hyperbolic}.  

Whenever $\dim_\Q H^*(X; \Q) < \infty$, there is an integer $n_X$, called the \emph{formal dimension} of $X$, such that $H^{n_X}(X; \Q)\neq 0$ and $H^j(X; \Q) = 0$, for all $j > n_X$.  If $X$ is a closed, orientable manifold, then clearly $n_X = \dim(X)$.

If $X$ is a rationally elliptic space, 
then the dimensions $d_i(X)$ of the rational homotopy groups of $X$ satisfy, among others, the relations
\beq
\label{E:degrels}
n_X \geq \sum_{i \in \N} 2i \, d_{2i}(X) 
\ \ \text{ and } \ \ 
n_X =  - d_2(X) +  \sum_{i = 2}^{\infty} (2i - 1)(d_{2i-1}(X) - d_{2i}(X)) .
\eeq

From the
homotopy groups, one constructs a graded vector space $V_X = \bigoplus_{i = 0}^\infty V^i$ associated to $X$, where $V^0 = \Q$, $\dim_\Q V^1 = \rank(\pi_1(X))$ and, for $i \geq 2$,
\[
V^i \cong \textrm{Hom}(\pi_i(X),\Q)\cong \pi_i^\Q(X) \cong \Q^{d_i(X)}.
\]
Clearly, $V^1 = 0$ whenever $\pi_1(X)$ is a finite (nilpotent) group.  
An element $v \in V^i$ is said to be \emph{homogeneous} of \emph{degree} $\deg(v) = i$.  

The tensor algebra $TV_X$ on $V_X$ has an associative multiplication, with a unit $1 \in V^0$, given by the tensor product $T^i V_X \ox T^j V_X \to T^{i + j} V_X$, where $T^k V_X = V_X^{\ox k}$.  Taking the quotient of $TV_X$ by the ideal  generated by the elements $v \ox w - (-1)^{ij} w \ox v$, where $\deg(v) = i$, $\deg(w) = j$, yields the \emph{free commutative graded algebra} $\wedge V_X$.  In particular, multiplication in $\wedge V_X$ satisfies $v \cdot w = (-1)^{ij} w \cdot v$, for all $v \in V^i$ and $w \in V^j$.

Given a homogeneous basis $\{v_1, \dots, v_N \}$ of $V_X$, set $\wedge(v_1, \dots, v_N) = \wedge V_X$.  We denote the linear span of elements $v_{i_1} v_{i_2} \cdots v_{i_q} \in \wedge V_X$, $1 \leq i_1 \leq i_2 \leq \dots \leq i_q \leq N$, of word-length $q$ by $\wedge^q V_X$.  
Define $\wedge^+ V_X = \bigoplus_{q \geq 1} \wedge^q V_X$.

The graded algebra $\wedge V_X$ has a linear \emph{differential} $d_X$, i.e. a linear map $d_X : \wedge V_X \to \wedge V_X$ satisfying the following properties:
\begin{itemize}
\item[(1)] \label{L:deg} 
$d_X$ has degree $+1$, i.e. $d_X$ maps elements of degree $i$ to elements of degree $i+1$.\vspace{.2cm}
\item[(2)] $d_X^2 = 0$.\vspace{.2cm}
\item[(3)] $d_X$ is a derivation, i.e. $d_X (v \cdot w) = d_X (v)\cdot w + (-1)^{\deg(v)} v \cdot d_X(w)$.\vspace{.2cm}
\item[(4)] \label{L:nil}
$d_X$ is nilpotent, i.e. there is an increasing sequence of graded subspaces $V(0) \In V(1) \In \cdots $ such that $V = \cup_{k=0}^\infty V(k)$, $d_X|_{V(0)} \equiv 0$ and $d_X \colon V(k) \to \wedge V(k-1)$, for all $k \geq 1$.
\end{itemize}
In addition, $d_X$ satisfies:
\begin{itemize}
\item[(5)] $d_X$ is decomposable, i.e. $\im(d_X) \In \wedge^{\geq 2} V_X$.
\end{itemize}

Since $d_X$ is a derivation, it clearly depends only on its restriction to $V_X$.  The pair $(\wedge V_X, d_X)$ is called the \emph{minimal model} for $X$ and its corresponding (rational) cohomology satisfies $H^*(\wedge V_X, d_X) = H^*(X; \Q)$.

The minimal models of a nilpotent space $X$ and its universal cover $\wt X$ are related as follows.  If $(\wedge V_{\wt X}, d_{\wt X})$ and $(\wedge W, d)$ denote the minimal models of $\wt X$ and the classifying space $B_G$ of $G= \pi_1(X)$, respectively, then $W = W^1$, $V_{\wt X}^1 = 0$ and the minimal model of $X$ is given by
$$
(\wedge V_X, d_X) = (\wedge W \ox \wedge V_{\wt X}, d_X) = (\wedge (W \oplus V_{\wt X}), d_X),
$$
where $d_X|_{\wedge W} = d$ and $d_X(v) - d_{\wt X}(v) \in \wedge^+ W \ox \wedge V_{\wt X}$ for all $v \in V_{\wt X}$.

By a slight abuse of terminology, two nilpotent 
spaces $X$ and $Y$ will be said to be \emph{rationally homotopy equivalent} (denoted $X \simeq_\Q Y$) if their minimal models are isomorphic, i.e. if there is a linear isomorphism $f\colon\wedge V_X \to \wedge V_Y$ which respects the grading and satisfies $f \circ d_X = d_Y \circ f$ and $f(v \cdot w) = f(v)\cdot f(w)$.  It is important to note that, first, it is not assumed that $\pi_1(X) \cong \pi_1(Y)$ and, second, the isomorphism $f$ is not necessarily induced by a map between $X$ and $Y$.  In fact, $X \simeq_\Q Y$ if and only if there is a chain of maps $X \to Y_1 \leftarrow Y_2 \to \cdots \leftarrow Y_s \to Y$ such that the induced maps on rational cohomology are all isomorphisms.  Observe that $X$ and $Y$ have isomorphic rational homotopy and rational cohomology groups whenever $X  \simeq_\Q Y$.

A nilpotent space $X$ with minimal model $(\wedge V_X, d_X)$ is said to be \emph{formal} if there is a morphism
$$
(\wedge V_X, d_X) \to (H^*(X; \Q), 0)
$$
of differential graded algebras inducing an isomorphism in cohomology.  If formal spaces $X$ and $Y$ have isomorphic rational cohomology rings, then $X \simeq_\Q Y$.  On the other hand, there are examples of nilpotent spaces $Y$ with rational cohomology ring isomorphic to that of a formal space $X$ and yet $X \not\simeq_\Q Y$; see, for example, \cite[Section 7]{Lu}.  A nilpotent space $X$ is \emph{intrinsically formal} if every nilpotent space $Y$ with rational cohomology ring isomorphic to $H^*(X;\Q)$ satisfies $X \simeq_\Q Y$ and, hence, is formal; that is, there is a unique rational homotopy type (minimal model) associated to the cohomology ring $H^*(X;\Q)$.  In particular, a product of spheres is intrinsically formal \cite{FH1}.

\section{Examples of double disk bundles}
\label{S:Examples}

Many interesting geometric examples have arisen via double disk bundle constructions.  In the hope of achieving a deeper understanding of the topological implications of certain geometric conditions, it is then natural to investigate the prevalence of manifolds admitting a double disk-bundle decomposition.  To this end, recall that a smooth, effective action of a compact Lie group $G$ on a smooth  manifold $M$ is of \emph{cohomogeneity one} if the orbit space $M^* = M/G$ of the action is one dimensional or, equivalently, if there is a $G$-orbit of codimension one. Alternatively, if 
the fixed-point set of the action of $G$ on $M$ is non-empty and has a component of codimension one in $M^*$, it is said to be \emph{fixed-point homogeneous}.


\begin{proposition}
\label{P:SuffCond}
A smooth, closed, simply connected manifold $M$ admits a double disk-bundle decomposition if at least one of the following conditions holds:
\begin{enumerate}[(a)]
\item \label{i:conn_sum}
$M$ is a connected sum of two  compact, rank-one symmetric spaces. 

\item  \label{i:cohom1}
$M$ admits a smooth, effective action of cohomogeneity one.

\item \label{i:BiqFol}
$M$ is the quotient of a cohomogeneity-one manifold by a free subaction.

\item \label{i:base}
$M$ is the total space of a smooth fiber bundle over a manifold which admits a double disk-bundle decomposition.

\item  \label{i:fiber}
$M$ is the total space of a linear sphere bundle 
admitting a smooth section.

\item \label{i:FPH}
$M$ admits a Riemannian metric with non-negative sectional curvature which is invariant under an isometric fixed-point-homogeneous action.
\end{enumerate}
\end{proposition}


\begin{proof}
\eqref{i:conn_sum} 
This is a simple consequence of the standard fact that  
removing a point from a non-spherical, simply connected, compact, rank-one symmetric space yields a disk bundle over a lower-dimensional compact, rank-one symmetric space. 

\eqref{i:cohom1} 
In this case, the statement is well known and follows from the Slice Theorem and fundamental group considerations (see, for example, \cite[Section 1]{Ho} and \cite[Section 1]{GWZ}).
In particular, if $G$ acts on $M$ with cohomogeneity one, then there are closed subgroups $H \In K_\pm \In G$ with $K_\pm / H \cong \sph^{\ell_\pm}$ and such that $M$ is equivariantly diffeomorphic to the union of the disk bundles $G \x_{K_\pm} D^{\ell_\pm + 1}$ glued (equivariantly) along their common boundary $G \x_{K_\pm} \sph^{\ell_\pm} \cong G/H$.

\eqref{i:BiqFol}  
Suppose $G$ acts on $M'$ with cohomogeneity one and that there is a subgroup $U \In G$ which acts freely on $M'$ with quotient $M$.  Observe first that the $U$ action on $M'$ preserves the orbits of the $G$ action.  Now, via the equivariant diffeomorphism mentioned in the proof of \eqref{i:cohom1} above, $U$ acts freely on each of the disk bundles $G \x_{K_\pm} D^{\ell_\pm + 1}$ by the action induced from the action of $U$ by left multiplication on 
the first factor of the product $G \x D^{\ell_\pm + 1}$.  As the $U$ action commutes with the action of $K_\pm$ on the right of the first factor, it follows that $U \bs (G \x_{K_\pm} D^{\ell_\pm + 1})$ is diffeomorphic to $(U \bs G) \x_{K_\pm} D^{\ell_\pm + 1}$.  These disk bundles both have boundary diffeomorphic to the biquotient $U \bs G / H$ and the equivariant gluing map in the double disk-bundle decomposition of $M'$ now induces a gluing of the quotient disk bundles, yielding the desired double disk-bundle decomposition of $M$.

\eqref{i:base} 
Suppose $Y \to M \to N$ is a fiber bundle such that $N$ is diffeomorphic to a double disk bundle $DB_- \cup_L DB_+$.  By \cite{QT}, there is a Riemannian metric $g_N$ on $N$ yielding a singular Riemannian foliation with singular leaves $B_\pm$ and regular leaf diffeomorphic to $L$.  If $(g_x)_{x \in N}$ is a smoothly varying family of Riemannian metrics on the fibers (that is, on $Y$), then a standard partition-of-unity argument yields a (unique) complete Riemannian metric $g_M$ on $M$ inducing the metric $g_x$ on the fiber $Y_x$, for each $x \in N$, and such that the projection map $(M, g_M) \to (N, g_N)$ is a Riemannian submersion.

On the other hand, it is well known
that, by pulling back the leaves of the foliation on the base, a singular Riemannian foliation can be lifted via a Riemannian submersion and, moreover, the codimensions of the leaves are preserved.  Therefore, $(M, g_M)$ admits a codimension-one singular Riemannian foliation with two singular leaves and, hence, a double disk-bundle decomposition (see, for example, \cite{Bo}).

\eqref{i:fiber} 
Suppose that $\sph^k \to M \to N$ is a linear sphere bundle admitting a smooth section $\sigma : N \to M$.  This can be viewed as the unit-sphere subbundle of a rank-$(k+1)$ vector bundle $\pi : E \to N$ equipped with a smooth fiberwise inner product $\< \,, \, \>$.  Therefore, $M$ can be decomposed as the union of the disk bundles $M_- = \cup_{x \in N} \{v \in \sph^k_x \mid \<v, \sigma(x)\> \leq 0 \}$ and $M_+ = \cup_{x \in N} \{v \in \sph^k_x \mid \<v, \sigma(x)\> \geq 0 \}$ over $N$.

\eqref{i:FPH} 
This assertion is taken directly from the Ph.D.~thesis  of Spindeler \cite{Sp}. 
\end{proof}

Whereas the double disk-bundle decomposition in Proposition \ref{P:SuffCond}\eqref{i:cohom1} admits a natural codimension-one singular Riemannian foliation with homogeneous leaves, notice that the decomposition in \eqref{i:BiqFol} admits a codimension-one singular Riemannian foliation with biquotient leaves, all the while retaining many of the characteristics of a cohomogeneity-one manifold.  This breaking of symmetry should have many applications and, indeed, has already been applied in \cite{GKS1}. Furthermore, observe that such decompositions arise whenever one has a compact Lie group $G$ and closed subgroups $H \In K_\pm \In G \x G$ with $K_\pm / H \cong \sph^{\ell_\pm}$ and such that $K_\pm$ act freely on $G$ via the respective restrictions of the action 
$$
(G \x G) \x G \to G \,;\, ((g_1, g_2), g) \mapsto g_1 \, g \, g_2^{-1} \,.
$$
This observation follows easily from the well-known diffeomorphism $G \cong \Delta G \bs (G \x G)$, where $\Delta G$ is the diagonal subgroup in $G \x G$, after first constructing a cohomogeneity-one $(G \x G)$-manifold with the given data and then applying Proposition \ref{P:SuffCond}\eqref{i:BiqFol} to the free $\Delta G$ subaction.

Manifolds admitting double disk-bundle decompositions arise frequently in geometry, as the following examples illustrate.


\begin{lem}
\label{L:GroupDDBD}
Let $G$ be a compact, connected Lie group which is not isomorphic to a finite quotient of a product $\prod_{i=1}^m G_i$, where $G_i \in \{\Eseven, \Eeight\}$ for all $i \in \{1, \dots, m\}$.
Then $G$ admits a cohomogeneity-one action and, hence, a double disk-bundle decomposition.
\end{lem}


\begin{proof}  
Recall that every compact, connected Lie group $G$ is isomorphic to the quotient $G'/\Gamma$ of a product $G' = T^k \x \prod_{i=1}^m G_i$ by a finite subgroup $\Gamma$ of the center of $G'$, where $T^k$ is a torus of rank $k$ and each $G_i$ is a simply connected, compact, simple Lie group.  In particular, if $K' \In G' \x G'$ acts effectively on $G'$ with cohomogeneity one, then it commutes with the action of $\Gamma$ and induces an effective cohomogeneity-one $K'' = K'/(K' \cap \Gamma)$ action on $G'/\Gamma$.  The only possible quotient spaces under this action are a closed interval and a circle.  In the first case, it follows as in Proposition \ref{P:SuffCond}\eqref{i:cohom1} that $G'/\Gamma$ admits a double disk-bundle decomposition.  In the case that the quotient space is a circle, then all $K''$ orbits are principal $G'/\Gamma$ and the quotient map $G'/\Gamma \to \sph^1$ is a bundle projection map.  In particular, it now follows from Proposition \ref{P:SuffCond}\eqref{i:base} that $G'/\Gamma$ admits a double disk-bundle decomposition.  By making use of the isomorphism, it is clear that in each case $G$ also admits a cohomogeneity-one action and, hence, a double disk-bundle decomposition.

It remains only to demonstrate that there is a cohomogeneity-one action on each possible product group $G' = T^k \x \prod_{i=1}^m G_i$.  If there is some $i_0 \in \{1, \dots, m\}$ such that $G_{i_0} \not\in \{\Eseven, \Eeight\}$, then the statement follows immediately from the classification by Kollross of cohomogeneity-one actions on compact, simple Lie groups \cite{Ko1, Ko2}.  Indeed, if $i_0 = m$, for example, then there is a subgroup $H_m \In G_m \x G_m$ acting on $G_m$ by cohomogeneity one and, therefore, the group $K' = T^k \x \prod_{i=1}^{m-1} G_i \x H_m$ acts on $G'$ with cohomogeneity one, as desired.

On the other hand, if $G' = T^k \x \prod_{i=1}^m G_i$, with $k > 0$ and $G_i \in \{\Eseven, \Eeight\}$ for all $i \in \{1, \dots, m\}$, then it is clear that $K' = T^{k-1} \x \prod_{i=1}^m G_i$ acts on $G'$ with cohomogeneity one and quotient space $\sph^1$.
\end{proof}

In \cite{Ko1} and \cite{Ko2}, Kollross made the additional observation that the simple Lie groups $\Eseven$ and $\Eeight$, when equipped with a bi-invariant metric, do not admit any isometric action of cohomogeneity one.  More generally, it is currently unknown whether $\Eseven$ and $\Eeight$ even admit a double disk-bundle decomposition.  As compact Lie groups with bi-invariant metrics are the simplest examples of Riemannian manifolds with non-negative sectional curvature, this suggests Grove's Double Soul Conjecture \cite{Gr} is quite subtle and difficult.  On the other hand, as remarked in \cite{Gr}, the situation appears to be better in the case of positive curvature.


\begin{thm}
\label{T:possec}
Every known example of a manifold admitting positive sectional curvature admits a double disk-bundle decomposition.
\end{thm}

\begin{proof} As described in \cite{Zi}, the known examples of closed, simply connected Riemannian manifolds with positive sectional curvature comprise compact rank one symmetric spaces (CROSSes), an infinite family of Eschenburg spaces in dimension $7$ and an infinite family of Bazaikin spaces in dimension $13$, as well as the following sporadic examples:  the homogeneous flag manifolds $\SU(3) / T^2$, $\Sp(3)/\Sp(1)^3$ and $F_4/\Spin(8)$; the (homogeneous) Berger space $\SO(5)/\SO(3)_{\mathrm{max}}$, where the embedding $\SO(3) \to \SO(3)_\mathrm{max} \In \SO(5)$  is induced from the unique irreducible $5$-dimensional representation of $\SO(3)$; a biquotient $\SU(3)\bq T^2$ (the inhomogeneous flag); and a cohomogeneity-one manifold $P_2$. 

It is well-known that the CROSSes admit smooth cohomogeneity-one actions, as does $P_2$, by construction.  Each of the homogeneous flag manifolds can be written as a linear sphere bundle over a CROSS, while the inhomogeneous flag $\SU(3) \bq T^2$ is the total space of a linear $\sph^2$-bundle over $\CP^2$ and it was shown in \cite{GKiS} that $\SO(5)/\SO(3)_{\mathrm{max}}$ is diffeomorphic to a linear $\sph^3$-bundle over $\sph^4$.  Therefore, it follows from Proposition \ref{P:SuffCond}{\color{blue} (d)} that each admits a double disk-bundle decomposition.

On the other hand, in general, the Eschenburg spaces $\SU(3)\bq S^1_{\ul p, \ul q}$ and the Bazaikin spaces $\SU(5) \bq (\Sp(2)\cdot S^1_{\ul q})$ neither admit a cohomogeneity-one action nor appear as the total space of a nice fiber bundle.  Nevertheless, the free quotient action in each case is (or, at least, can be rewritten as) a subaction of a cohomogeneity-one action on $\SU(3)$ or $\SU(5)$, respectively.  By Proposition \ref{P:SuffCond}\eqref{i:BiqFol}, it then follows that each admits a double disk-bundle decomposition.
\end{proof}


\section{Topology of double disk bundles}
\label{S:top}

Given the relative simplicity of the construction, it is possible to say quite a lot about the topology of double disk bundles.  Some useful results in this regard are collected here.

{ To begin, observe that the cylinder $D^1 \x \sph^1$ is a disk bundle over $\sph^1$ for which the boundary is disconnected, due to the fact that $\partial D^1 = \sph^0 \cong \{\pm 1\}$.  Hence, each component of the boundary of $D^1 \x \sph^1$ may be glued to (a component of) the boundary of a distinct disk bundle.  Therefore, it is possible that a closed manifold could decompose into more than two disk bundles.  For example, the sphere $\sph^2$ can be decomposed as the union of a chain of cylinders glued end to end and capped off by two disks $D^2 \x \{\mathrm{pt} \}$.  Of course, it is clear that capping off one end of the union of such a chain of cylinders yields a manifold diffeomorphic to the $2$-disk, so that the above decomposition of $\sph^2$ reduces to the union of two disks.  The following proposition shows that this reduction to a double disk-bundle decomposition is a general phenomenon.}



\begin{proposition}
\label{P:connected}
Let $M$ be a smooth, closed, connected manifold which can be decomposed as the union of disk bundles glued together via diffeomorphisms of the components of their respective boundaries. Then $M$ admits a double disk-bundle decomposition $DB_- \cup_f DB_+$ for which both $B_\pm$ are connected.
\end{proposition}


\begin{proof} 
Since $M$ is closed, it can be decomposed as the union of at most finitely many disk bundles.     Let $D^{\ell_i + 1} \to DB_i \to B_i$, $i = 1, \dots, m$, be the disk bundles in such a decomposition.  
As each $DB_i \In M$ is compact, it follows that each base manifold $B_i$ is closed.  Furthermore, it may be assumed without loss of generality that each $B_i$ is connected.  Let $\sph^{\ell_i} \to SB_i \to B_i$, $i = 1, \dots, m$, denote the corresponding sphere bundles.  The long exact homotopy sequence for $\sph^{\ell_i} \to SB_i \to B_i$ yields that $SB_i$ has at most two components, where $SB_i$ being disconnected implies that $\ell_i = 0$ and 
that $DB_i \cong B_i \x [-1, 1]$.

As $M$ is closed and connected, either all of the sphere bundles $SB_i$, $i =1, \dots, m$, are disconnected, or there are precisely two disk bundles with connected boundary.  After relabelling the disk bundles $DB_i$ and, if necessary, reparametrizing their fibers, it follows from the hypothesis that, if all $SB_i$ are disconnected, there are diffeomorphisms $f_i\colon B_i \x \{+1\} \to B_{i+1} \x \{-1\}$, for all $i \in \{1, \dots, m-1\}$, and $f_m \colon B_m \x \{+1\} \to B_1 \x \{-1\}$.  On the other hand, if $SB_1$ and $SB_m$ are connected (and $m > 2$), then the diffeomorphisms $f_1$ and $f_{m-1}$ may be replaced by $f_1\colon SB_1 \to B_2 \x \{-1\}$ and $f_{m-1}\colon B_{m-1} \x \{+1\} \to SB_m$, respectively, while $f_m$ does not occur.

It is, however, well known that in both cases $DB_1 \cup_{f_1} \dots \cup_{f_{m-2}} DB_{m-1}$ is diffeomorphic to $DB_1$, independent of the choices of diffeomorphisms $f_1, \dots, f_{m-2}$; see, for example, \cite[Chapter VI, Section 5]{Kos}.  Hence, there is always a diffeomorphism $f \colon SB_1 \to SB_m$ such that $M$ is diffeomorphic to $DB_1 \cup_f DB_m$, as desired.
\end{proof}

Recall that, if a manifold $M$ is the total space of a fiber bundle over $\sph^1$, then $M$ has infinite fundamental group and, by Proposition \ref{P:SuffCond}\eqref{i:base}, it admits a double disk-bundle decomposition $(B_- \x [-1, 1]) \cup_f (B_+ \x [-1,1])$ with $B_- \cong B_+$ being of codimension one.  In fact, the permissible codimensions of $B_\pm$ in a double disk bundle $M = DB_- \cup_f DB_+$ are always restricted by the fundamental group of $M$.


\begin{proposition}
\label{P:exceptional}  
Let $M$ be a smooth, closed, simply connected manifold which admits a double disk-bundle decomposition  $DB_- \cup_L DB_+$ with $B_\pm$ connected.  Then $B_\pm$ are both of codimension $\geq 2$.
\end{proposition}


\begin{proof}Suppose, without loss of generality, that $B_- \In M$ is of codimension one and let $\pi_- : DB_- \to B_-$ denote the bundle projection map.  Thus, the fiber $\pi_-^{-1}(b) \In DB_-$ over a point $b \in B_- $ is diffeomorphic to an interval and intersects $B_-$ transversally in a single point.  Moreover, if the two points comprising $\pi_-^{-1}(b) \cap L$ are joined by an arbitrary curve $c_+$ in $DB_+$, then the closed curve $c_b = \pi_-^{-1}(b) \cup c_+$ in $M$ intersects the closed submanifold $B_-$ transversally in a single point.   Consequently, the intersection form
\[
I_M \colon H_{n-1}(M) \x H_1(M) \to \Z
\]
yields $I_M([B_-], [c_b]) = \pm 1$, where $[B_-]$ and $[c_b]$ are the homology classes represented by $B_-$ and $c_b$, respectively.  However, this is a contradiction, since $M$ being simply connected implies that $H_1(M) = 0$ and, hence, that the intersection form $I_M$ is trivial.  Therefore, the codimension of $B_-$ must be at least two.
\end{proof}

Note that an analogue of the above proposition for closed, smooth, simply connected, cohomogeneity-one manifolds appeared in \cite[Lemma 1.6]{GWZ}.

The following characterization of trivial orientable circle bundles is often useful when dealing with double disk-bundle decompositions where at least one of the singular leaves is of codimension two. In the sequel, an element of a finitely genereated abelian group $A$ will be called a \emph{generator} if it is neither torsion nor a non-trivial multiple of any other element.  Equivalently, an element of $A$ will be called a generator if it generates a free abelian subgroup of rank one which is not properly contained in any other free abelian subgroup of rank one.


\begin{thm}  
\label{T:triv}
Let $\sph^1 \to L\to B$ be an orientable circle bundle over a connected manifold $B$ with $\pi_1(L)$ abelian.  Then the bundle is trivial if and only if the induced homomorphism $\pi_1(\sph^1) \to \pi_1(L)$ is injective with image containing a generator of $\pi_1(L)$.
\end{thm}


\begin{proof}
Recall, for example, from \cite[Prop.\ 6.15]{Mo}, that every orientable circle bundle $\sph^1 \to L\to B$ is principal and, hence, classified by its Euler class $e \in H^2(B)$.  In particular, $\sph^1 \to L\to B$ is trivial if and only if $e = 0$.  

The Gysin sequence corresponding to $\sph^1 \to L\to B$ yields an exact sequence
$$
0 \to H^1(B) \to H^1(L) \xrightarrow{f} H^0(B) \xrightarrow{\smile e} H^2(B) \to \cdots
$$

Let $i : \sph^1 \to L$ be inclusion of a circle fiber and let $j : (D^2, \sph^1) \to (E, L)$ be the corresponding fiber inclusion of pairs, where $D^2 \to E \to B$ is the disk bundle with boundary $\sph^1 \to L\to B$.  There is a commutative diagram (see, for example, \cite[Section 4D]{Ha2})
$$
\xymatrix{
H^1(\sph^1) \ar[d]_\cong & H^1(L) \ar[l]_{i^*} \ar[d] \ar[r]^f & H^0(B) \ar[dl]^\Phi \\
H^2(D^2, \sph^1) & H^2(E, L) \ar[l]^{j^*} & 
}
$$
where the vertical maps are those in the long exact sequences for the pairs, $j^*$ is an isomorphism, $\Phi$ is the Thom isomorphism and $f$ is the map in the Gysin sequence above.  Therefore, the above Gysin sequence can be modified to yield a commutative diagram
\beq
\label{E:modCD}
\xymatrix{
0 \ar[r] & H^1(B) \ar[r] & H^1(L) \ar[r]^f \ar[dr]^{i^*} & H^0(B) \ar[r]^{\smile e} & H^2(B) \ar[r] & \cdots \\
& & & H^1(\sph^1) \ar[u]_\cong & &
}
\eeq
Clearly, therefore, the Euler class $e$ is trivial if and only if $i^* : H^1(L) \to H^1(\sph^1)$ is surjective.  But, since $H^1(\sph^1)$ is free abelian, $i^*$ is surjective if and only if $i_* : H_1(\sph^1) \to H_1(L)$ is injective and maps a generator of $H_1(\sph^1)$ to a generator of $H_1(L)$.  Since $\pi_1(L)$ is abelian, naturality in the Hurewicz Theorem now ensures that  $e = 0$ if and only if the induced homomorphism $i_* : \pi_1(\sph^1) \to \pi_1(L)$ is injective and maps a generator of $\pi_1(\sph^1)$ to a generator of $\pi_1(L)$, as desired.
\end{proof}

In \cite{GH}, Grove and Halperin systematically studied spaces admitting a double disk-bundle decomposition from the perspective of rational homotopy theory. For the present work, it is useful to have a summary of their results adapted to the current situation.


\begin{theorem}[Grove--Halperin \protect{\cite{GH}}]
	\label{T:HOM_FIBER}
Suppose that a smooth, closed, simply connected manifold $M$ admits a double disk-bundle decomposition $DB_- \cup_L DB_+$, where $B_\pm$ are both connected.  If $F$ denotes the homotopy fiber of the inclusion $L \to M$, then $L$ and $F$ are nilpotent spaces and 
$F$ is rationally rational homotopy equivalent to one of the spaces listed in Table \ref{table:Qtype}, where $A_m(4)$ denotes a certain simply connected topological space whose non-trivial rational homotopy groups are in degrees $4$, $7$ and $4m - 1$. Moreover, the possible fundamental groups of $F$ and codimensions of $B_\pm$ are indicated in Table \ref{table:Qtype}.
\end{theorem}


\begin{table}

\begin{tabular}{|Sc|Sc|Sc|Sc|}
\hline
\multirow{2}{*}{$\pi_1(F)$} & 
\multirow{2}{*}{$F \simeq_\Q$} & 
\multirow{2}{*}{$\{\alpha, \beta\} = \{\ell_\pm\}$} &  
\multicolumn{1}{>{\centering\let\newline\\\arraybackslash\hspace{0pt}}p{35mm}|}{Orientability of \newline  $\sph^{\ell_\pm}$-bundles} \\
\hline \hline

$\Z^2$ & $\sph^1 \x \sph^1 \x \Omega \sph^3$ & \multirow{4}{*}{$1 = \alpha = \beta$} & Both \\ 
\cline{1-2} \cline{4-4}

$\Z \oplus \Z_2$ & $\sph^1 \x \sph^3 \x \Omega \sph^5$ &  & One \\
\cline{1-2} \cline{4-4}

$Q_8$ & $\sph^3 \x \sph^3 \x \Omega \sph^7$ &  & Neither \\
\hline

\multirow{2}{*}{$\Z$} & $\sph^1 \x \sph^\beta \x \Omega \sph^{\beta + 2}$ & $1 = \alpha < \beta$ & Both \\
\cline{2-4}
& $\sph^1 \x \sph^{2\beta + 1} \x \Omega \sph^{2\beta + 3}$ & $1 = \alpha < \beta$, $\beta$ odd & $\sph^1$-bundle \\
\hline

\multirow{11}{*}{$0$} & $\sph^\alpha \x \sph^\beta \x \Omega \sph^{\alpha + \beta + 1}$ & $1 < \alpha \leq \beta$ &  \multirow{11}{*}{Both} \\ \cline{2-3}
& $\sph^\alpha \x \Omega \sph^{\alpha + 1}$ & $1 < \alpha = \beta$  &  \\ \cline{2-3}
& $\SU(3)/T^2 \x \Omega \sph^7$ & \multirow{4}{*}{$2 = \alpha = \beta$} &  \\ \cline{2-2}
& $\Sp(2)/T^2 \x \Omega \sph^9$ & &  \\ \cline{2-2}
& $\Gtwo/T^2 \x \Omega \sph^{13}$ & &  \\ \cline{2-3}
& $\Sp(3)/\Sp(1)^3 \x \Omega \sph^{13}$ & \multirow{4}{*}{$4 = \alpha = \beta$} &  \\ \cline{2-2}
& $A_4(4) \x \Omega \sph^{17}$ & &  \\ \cline{2-2}
& $A_6(4) \x \Omega \sph^{25}$ & &  \\ \cline{2-3}
& $\Ffour/\Spin(8) \x \Omega \sph^{25}$ & $8 = \alpha = \beta$ &  \\ \hline

\end{tabular}

\vspace{10pt}
\caption{Properties of the homotopy fiber $F$ and the bundles $\sph^{\ell_\pm} \to L \to B_\pm$ associated to a double disk-bundle decomposition $DB_- \cup_L DB_+$}
\label{table:Qtype}

\end{table}

Observe from the long exact homotopy sequence for the homotopy fibration $F \to L \to M$ that, in particular, $\pi_1(L)$ must be abelian whenever $\rank(\pi_1(L)) \geq 1$.  As a first, simple application of Theorem \ref{T:HOM_FIBER}, one obtains a criterion for 
a double disk bundle to be rationally elliptic.


\begin{lemma} 
\label{L:high_enough}
Let $M$ be a smooth, closed, simply connected manifold which admits a double disk-bundle decomposition $DB_- \cup_L DB_+$ with $B_\pm$ connected.  If there exists a $j_0 \in \N$ such that the rational homotopy groups of some $X \in \{L, B_\pm\}$
satisfy $\pi_j^\Q(X) = 0$, for all $j \geq j_0$, then $M$ is rationally elliptic.
\end{lemma}


\begin{proof}
In either case, since $L$ is a sphere bundle over $B_\pm$, all rational homotopy groups $\pi_j^\Q(L)$ of $L$ must vanish whenever $j \geq j_0$, for some $j_0 \in \N$. Let  $F$ be the homotopy fiber of the inclusion map $L \to M$. By Theorem \ref{T:HOM_FIBER},  
the rational homotopy groups of $F$ vanish in sufficiently high dimensions.  The long exact homotopy sequence for the homotopy fibration $F\to L \to M$ now yields that there is some $j_1 \in \N$ such that $\pi_j^\Q(M) = 0$, for all $j \geq j_1$, and, hence, that $M$ is rationally elliptic.
\end{proof}

If a manifold $M$ of arbitrary dimension admits a double disk-bundle decomposition $DB_- \cup_L DB_+$ with a singular leaf $B \in \{B_\pm\}$ of sufficiently low dimension, it turns out that $M$ is always rationally elliptic.


\begin{proposition}
\label{P:L_BOUND}
Suppose that $M$ is a smooth, closed, simply connected manifold which admits a double disk-bundle decomposition $DB_- \cup_L DB_+$ with $B_\pm$ connected and $\dim(B)\leq 3$, for some $B \in \{B_\pm\}$.  Then $M$ is rationally elliptic.
\end{proposition}


\begin{proof}
If $F$ is the homotopy fiber of the inclusion $L \to M$ then, since $M$ is simply connected, the long exact homotopy sequence associated to the homotopy fibration $F\to L \to M$ ensures that $\pi_1(F) \to \pi_1(L)$ is surjective.  From the list of possible fundamental groups of $F$ given in Table \ref{table:Qtype}, it follows that $\pi_1(L)$ is either $Q_8$ or abelian.  Considering the  sphere bundle $\sph^\ell \to L \to B$ corresponding to $B$, it is clear that $\pi_1(B)$ is itself either $Q_8$ or abelian.  As $\dim(B) \leq 3$, this implies that $B$ is finitely covered by one of $\sph^1$, $\sph^2$, $T^2$, $\sph^3$, $\sph^2\times \sph^1$ or $T^3$; see, for example, \cite[Table 2, p.\ 25]{AFW}.  In particular, by Lemma~\ref{L:high_enough}, $M$ is rationally elliptic.
\end{proof}

As there are well-known classifications of simply connected, closed manifolds in low dimensions, it is convenient to have a criterion which ensures that there is a singular leaf in a double disk-bundle decomposition whose universal cover is closed.


\begin{lemma}
\label{L:pi1B}
Suppose that $M$ is a smooth, closed, simply connected manifold which admits a double disk-bundle decomposition $DB_- \cup_L DB_+$ with $B_\pm$ connected.  Then $\rk(\pi_1(L)) \leq 1$ if and only if at least one of the singular leaves $B_\pm$ has finite fundamental group.
\end{lemma}


\begin{proof}

Assume that $\rk(\pi_1(L)) \leq 1$ and suppose that $\pi_1(B_\pm)$ are both infinite.  The long exact homotopy sequences for the bundles $\sph^{\ell_\pm} \to L \to B_\pm$ then yield 
\[
1 \leq \rk(\pi_1(B_\pm)) \leq \rk(\pi_1(L)) \leq 1
\]
and, hence, 
that $\pi_1(L)$ and its quotients $\pi_1(B_\pm)$ are abelian of rank one.  Furthermore, the images of the homomorphisms $\pi_1(\sph^{\ell_\pm}) \to \pi_1(L)$ in the respective long exact sequences are both finite and, hence, cannot together generate $\pi_1(L)$.  Since $\pi_1(M) = 0$, this is a contradiction to equation (3.7) of \cite{GH}. 

Assume, on the other hand, that there is some $B \in \{B_\pm\}$ with $\pi_1(B)$ finite.  Then, since $\rk(\pi_1(\sph^{\ell_\pm})) \leq 1$, the long exact homotopy sequences for $\sph^{\ell_\pm} \to L \to B_\pm$ ensure that $\rk(\pi_1(L)) \leq 1$, as desired.
\end{proof}


Recall that a topological space is said to be of \emph{finite type} if it is weakly homotopy equivalent to a CW-complex with finitely many $k$-cells for each $k$.  It turns out that, for a nilpotent space $X$, being of finite type is equivalent to the integral homology groups $H_j(X)$ being finitely generated for all $j \geq 1$, and to the homotopy groups $\pi_j(X)$ being finitely generated for all $j \geq 1$; see \cite[Theorem 4.5.2]{MP}. 

\begin{lem}
\label{L:lift_nil}
Suppose $X$ is a nilpotent space of finite type and that $p: \ol X \to X$ is a covering map.  Then $\ol X$ is also a nilpotent space of finite type.
\end{lem}

\begin{proof}
Recall that the map $p$ induces an injection on fundamental groups and an isomorphism on higher homotopy groups.  In particular, by the discussion immediately preceding the lemma, it thus suffices to show that $\ol X$ is a nilpotent space. As subgroups of nilpotent groups are nilpotent, the nilpotency of $\pi_1(X)$ ensures that $\pi_1(\ol X) \cong p_*(\pi_1(\ol X)) \In \pi_1(X)$ is nilpotent, while the isomorphisms $p_* : \pi_k(\ol X) \to \pi_k(X)$, $k \geq 2$, are, by definition, equivariant with respect to the action of $\pi_1(\ol X)$ (see \cite[pp.~341--342]{Ha2}); that is, 
$$
p_*(\gamma \cdot \vphi) = p_*(\gamma) \cdot p_*(\vphi)
$$
for all $\gamma \in \pi_1(\ol X)$ and all $\vphi \in \pi_k(\ol X)$.  It now follows easily from the nilpotency of the space $X$ that $\ol X$ is nilpotent.
\end{proof}

This lemma finds a useful application in the context of double disk bundles.  Indeed, the regular leaf, after taking an appropriate cover, behaves, up to homotopy, like a closed manifold of possibly lower dimension. 
Recall that the \emph{maximal free abelian cover} $\ol N$ of a closed, smooth, orientable manifold $N$ with first Betti number $b_1(N)$ is the total space of a principal $\Z^{b_1(N)}$-bundle over $N$.  Indeed, $\ol N$ is a smooth, orientable manifold, with finite fundamental group an extension of the torsion subgroup of $H_1(N)$ by the commutator subgroup of $\pi_1(N)$, and the group of deck transformations for the covering is $\Z^{b_1(N)}$.  For example, the maximal free abelian cover of a product $T^k \x N$ is $\R^k \x N$ whenever $\pi_1(N)$ is finite.


\begin{prop}
\label{P:PDmiddle}
Let $M$ be a smooth, closed, simply connected manifold which admits a double disk-bundle decomposition $DB_- \cup_L DB_+$ with $B_\pm$ connected.  Then the maximal free abelian cover $\ol L$ of $L$ is a rational Poincar\'e-duality nilpotent space of formal dimension $\dim(L) - b_1(L)$.
\end{prop}

\begin{proof}
Recall from Theorem \ref{T:HOM_FIBER} that $L$ is a nilpotent space.  Now, being a closed, smooth, codimension-one submanifold of a closed, simply connected manifold, $L$ must also be orientable (by \cite[p.~107]{Hi}) with $H_j(L)$ finitely generated for all $j \geq 1$.  Therefore, by \cite[Theorem 4.5.2]{MP} and Lemma \ref{L:lift_nil}, the maximal free abelian cover $\ol L$ of $L$ is a smooth, orientable manifold which is nilpotent and of finite type.  In particular, $\dim_\Q H_*(\ol L; \Q)$ must be finite dimensional.  It now follows from Milnor-Barge duality (see \cite{Barge}, \cite{Mi} and also \cite[Theorem 5.2]{KS}) that there is an integral homology class $[\ol L] \in H_{\dim{L} - b_1(L)}(\ol L)$ such that the cap product
$$
\frown [\ol L] : H^j(\ol L; \Q) \to H_{\dim(L) - b_1(L) - j} (\ol L: \Q)
$$
is an isomorphism for all $j \geq 0$, as desired.
\end{proof}

Observe that, if $F$ is the homotopy fiber of the inclusion $L \to M$, it follows from Table \ref{table:Qtype}, the Hurewicz Theorem and the long exact homotopy sequence for the homotopy fibration $F \to L \to M$ that $b_1(L) = \rank( \pi_1(L) ) \leq 2$.  In other words, the maximal free abelian cover $\ol L$ of the regular leaf $L$ behaves on the level of rational cohomology like a closed, simply connected manifold with $\dim(L) - 2 \leq \dim(\ol L) \leq \dim(L)$.

Just as it is convenient to know that the maximal free abelian cover of the regular leaf $L$ is nilpotent, it is often useful to have a criterion ensuring that a singular leaf is a nilpotent space.


\begin{lem}  
\label{L:nilp}
Suppose $\sph^1 \to L \xrightarrow{\pi} B$ is a circle bundle 
such that the induced homomorphism $\pi_1(\sph^1) \to \pi_1(L)$ is injective.  If $L$ is nilpotent, then so too is $B$.
\end{lem}


\begin{proof}
From the long exact homotopy sequece for the bundle, it is clear that the homomorphism $\pi_* : \pi_1(L) \to \pi_1(B)$ is surjective.  Being the image of a nilpotent group under a homomorphism, it follows that $\pi_1(B)$ is nilpotent.

Since $\pi_1(\sph^1)$ injects into $\pi_1(L)$, 
the long exact homotopy sequence 
yields isomorphisms $\pi_*:\pi_k(L)\rightarrow \pi_k(B)$ for all $k \geq 2$.  

For $k \geq 2$, let 
$$
\pi_k(L) = G^k_0 \rnorm G^k_1 \rnorm \cdots \rnorm G^k_m = \{1\}
$$
be the chain of subgroups associated to the abelian group $\pi_k(L)$ which witness the nilpotency of the action of $\pi_1(L)$.  Define subgroups $H^k_j = \pi_*(G^k_j) \In \pi_k(B)$ for $j \in \{0, \dots, m\}$.  As $\pi_*$ is an isomorphism, it is clear that, for each $j \in \{1, \dots, m\}$, $H_j$ is normal in $H_{j-1}$; that is, 
$$
\pi_k(B) = H^k_0 \rnorm H^k_1 \rnorm \cdots \rnorm H^k_m = \{1\}\,.
$$

It remains to show that the action of $\pi_1(B) = \pi_*(\pi_1(L))$ preserves each $H_j$ and induces a trivial action on each $H_{j-1}/H_j$.  By definition (see, for example, \cite[p.~341--342]{Ha2}), it is clear that $\pi_*:\pi_k(L)\rightarrow \pi_k(B)$ is equivariant with respect to the action of $\pi_1(L)$; that is, 
$$
\pi_*(\gamma \cdot \vphi) = \pi_*(\gamma) \cdot \pi_*(\vphi)
$$
for all $\gamma \in \pi_1(L)$ and all $\vphi \in \pi_k(L)$.  Since $\pi_1(B) = \pi_*(\pi_1(L))$, it now clearly follows that $H_j$ is closed under the action of $\pi_1(B)$ for all $j \in \{0, \dots, m\}$.

Finally, if $\gamma \in \pi_1(L)$ and $\vphi \in G_{j-1}$, it follows from the nilpotency of the $\pi_1(L)$ action that 
$$
\pi_*(\gamma) \cdot \pi_*(\vphi) = \pi_*(\gamma \cdot \vphi) \in \pi_*(\vphi \cdot G_j) = \pi_*(\vphi) H_j
$$
and, hence, that $\pi_1(B)$ acts trivially on each $H_{j-1}/H_j$.
\end{proof}


Recall now that the \emph{Lusternik--Schnirelmann category} $\cat(Y)$ of a topological  space $Y$ is defined to be the least integer $m \in \N$ such that $Y$ is the union of $m + 1$ open sets, each of which is contractible in $Y$.  The \emph{rational Lusternik--Schnirelmann category} $\cat_0(Y)$ of $Y$, on the other hand, is defined to be the minimal $\cat(Z)$ among all $Z$ which are rationally homotopy equivalent to $Y$.

The following theorem will be an important tool in the remainder of the paper.  Although our interest is restricted to the manifold case, the statement remains true in the setting of double mapping cylinders (see \cite{GH}) as long as $\cat_0(L)$ is finite.  In this way, it is a generalization of Lemma 6.3 of \cite{GH}.  See \cite[Prop.\ 2.7]{DVK} for another related statement under different hypotheses.

By Table \ref{table:Qtype}, 
the homotopy fiber $F$ of the inclusion map $L\to M$ always has a loop-space factor of the form $\Omega \sph^k$ for some $k \in \N$.  Denote by $s$ the degree of the unique non-trivial rational homotopy group of $\Omega \sph^k$ of even degree.  That is, $s = k-1$, if $k$ is odd, and $s = 2(k-1)$, if $k$ is even.


\begin{thm}
\label{T:NoRank}
Suppose that $M$ is a smooth, closed, simply connected manifold which admits a double disk-bundle decomposition $DB_- \cup_L DB_+$ with $B_\pm$ connected. Then, in the long exact sequence of rational homotopy groups associated to the homotopy fibration $F\to L \to M$, the connecting homomorphism $\partial\colon\pi_{s+1}^\Q(M)\to \pi_s^\Q(F)$ is non-trivial.
\end{thm}


\begin{proof}
Suppose that the homomorphism $\partial$ is trivial.  In particular, it then follows from the long exact sequence that the map $\pi_s^\Q(F) \to \pi_s^\Q(L)$ is injective.  

Consider now the space $W_{s-1}$ in the Whitehead tower $\cdots \to W_{2} \to W_{1} \to W_{0} \to F$ 
associated to $F$, that is, an $(s-1)$-connected space such that the map $W_{s-1} \to F$ induces an isomorphism $\pi_j(W_{s-1}) \to \pi_j(F)$ for every $j \geq s$.  Therefore, by Table \ref{table:Qtype}, there are three possible configurations of non-trivial rational homotopy groups for $W_{s-1}$.  First, if $\ell_\pm$ are even with $\ell_- = \ell_+$ and $F \simeq_\Q \sph^{\ell_-} \x \Omega \sph^{\ell_- + 1}$, then $\pi_1(F) = 0$ and
\[
\pi_j^\Q(W_{s-1}) = 
\begin{cases}
\Q^2, & j = s = \ell_-,\\
\Q, &j = 2\ell_- - 1.
\end{cases}
\]
Second, if $\ell_\pm$ are even and $\ell_- \neq \ell_+$, then $F \simeq_\Q \sph^{\ell_-} \x \sph^{\ell_+} \x \Omega \sph^{\ell_-  + \ell_+ + 1}$, $\pi_1(F) = 0$ and
\[
\pi_j^\Q(W_{s-1}) = 
\begin{cases}
\Q, & j = s = \ell_- + \ell_+,\\
\Q, &j = 2\max\{\ell_\pm\} - 1 > s.
\end{cases}
\]
In all other cases, the only non-trivial rational homotopy group is $\pi_s^\Q(W_{s-1}) = \Q$.

In all three scenarios, by computing the minimal model it becomes clear that there is an element $x \in H^s(W_{s-1}; \Q) \neq 0$ such that $x^m \neq 0$ for all $m \in \N$.  In particular, this implies that $W_{s-1}$ has cup-length $\mathrm{cup}(W_{s-1}) = \infty$.  Moreover, from Propositions 27.14 and 28.1 of \cite{FHT} it now follows that
\[
\cat_0(W_{s-1}) \geq \mathrm{cup}(W_{s-1}) = \infty.
\]

On the other hand, the composition $W_{s-1} \to F \to L$ induces (by assumption) an injection $\pi_\mathrm{even}^\Q(W_{s-1}) \to \pi_\mathrm{even}^\Q(L)$, while the kernel of $\pi_\mathrm{odd}^\Q(W_{s-1}) \to \pi_\mathrm{odd}^\Q(L)$ has dimension $\kappa \in \{0, 1\}$.  By Theorem \ref{T:HOM_FIBER}, $L$ (and $F$) is nilpotent.  Therefore, the Mapping Theorem \cite[Theorem 2.81]{FOT} yields $\cat_0(W_{s-1}) \leq \cat_0(L)$ whenever $\kappa = 0$.  By combining Propositions 27.2 and 27.5 and Lemma 28.2 of \cite{FHT}, it may thus be concluded in the case $\kappa = 0$ that 
\[
\infty = \cat_0(W_{s-1}) \leq \cat_0(L) \leq  \cat(L) \leq n-1 < \infty,
\]
a contradiction.

Suppose, therefore, that $\kappa = 1$, hence, that $\pi_\mathrm{odd}^\Q(W_{s-1}) \neq 0$.  Observe first that the Mapping Theorem applied to the map $W_{s-1} \to F$ yields
\[
\infty = \cat_0(W_{s-1}) \leq \cat_0(F).
\]

In the case that $\ell_\pm$ are even with $\ell_- = \ell_+$ and $F \simeq_\Q \sph^{\ell_-} \x \Omega \sph^{\ell_- + 1}$, the unique non-trivial rational homotopy group in even degrees is $\pi_s^\Q(F) = \Q^2$, which, by assumption, injects into $\pi_s^\Q(L)$.  Therefore, $\pi_\mathrm{even}^\Q(F) \to \pi_\mathrm{even}^\Q(L)$ is injective, while the kernel of 
\[
\pi_\mathrm{odd}^\Q(F) = \pi_{2\ell_- - 1}^\Q(F) = \Q \to \pi_\mathrm{odd}^\Q(L)
\] 
is $1$-dimensional (since $\kappa = 1$).  Hence, Theorem II of \cite{FH2} implies that
\[
\infty = \cat_0(F) \leq \cat_0(L) + 1 \leq n < \infty,
\]
again a contradiction.

Finally, suppose that ($\kappa = 1$ and) $\ell_\pm$ are even with $\ell_- \neq \ell_+$, that is, $F \simeq_\Q \sph^{\ell_-} \x \sph^{\ell_+} \x \Omega \sph^{\ell_-  + \ell_+ + 1}$.  In particular, $F$ has exactly five non-trivial rational homotopy groups, occurring in degrees 
\[
\min\{\ell_\pm\} < \max\{\ell_\pm\}, \ 2\min\{\ell_\pm\} - 1 < s = \ell_- + \ell_+ < 2 \max\{\ell_\pm\} - 1
\]
and each of rank $1$.  Therefore, the kernel of $\pi_\mathrm{odd}^\Q(F) = \Q^2 \to \pi_\mathrm{odd}^\Q(L)$ has dimension $\in \{1,2\}$ (since $\kappa = 1$).  Moreover, as mentioned above, $F$ and, hence, $L$ are simply connected.  If $\pi_\mathrm{even}^\Q(F) \to \pi_\mathrm{even}^\Q(L)$ is injective, then Theorem II of \cite{FH2} yields a contradiction
\[
\infty = \cat_0(F) \leq \cat_0(L) + 1 \leq n + 1< \infty
\]
as before.  Therefore, given the assumption that $\pi_s^\Q(F) \to \pi_s^\Q(L)$ is injective, it remains only to show that $\pi_{\ell_\pm}^\Q(F) = \Q \to \pi_{\ell_\pm}^\Q(L)$ are injective.  It clearly suffices to show that these homomorphisms are non-trivial.

To this end, observe that the inclusion $L \to M$ factors through the disk bundles $DB_\pm$, and that the inclusions $L \to DB_\pm$ and $DB_\pm \to M$ are homotopic to the corresponding sphere-bundle projection maps $L \to B_\pm$ and inclusions $B_\pm \to M$ respectively.  Therefore, the homomorphism $\pi_{\ell_\pm}^\Q(L) \to \pi_{\ell_\pm}^\Q(M)$ decomposes as a composition $\pi_{\ell_\pm}^\Q(L) \to \pi_{\ell_\pm}^\Q(B_\pm) \to \pi_{\ell_\pm}^\Q(M)$.

Now, if $\pi_{\ell_\pm}^\Q(F) = \Q \to \pi_{\ell_\pm}^\Q(L)$ is trivial, then the long exact sequence for $F \to L \to M$ implies that $\pi_{\ell_\pm}^\Q(L) \to \pi_{\ell_\pm}^\Q(M)$ is injective, which, by the above observations, further implies that $\pi_{\ell_\pm}^\Q(L) \to \pi_{\ell_\pm}^\Q(B_\pm)$ is injective.  However, since $\cat_0(\sph^{\ell_\pm}) = 2$, Theorem II of \cite{FH2} applied to the sphere bundle $\sph^{\ell_\pm} \to L \to B_\pm$ reveals that $\pi_{\ell_\pm}^\Q(\sph^{\ell_\pm}) = \Q \to \pi_{\ell_\pm}^\Q(L)$ must be injective and, hence, by exactness, that $\ker(\pi_{\ell_\pm}^\Q(L) \to \pi_{\ell_\pm}^\Q(B_\pm)) \neq 0$, a contradiction.  Thus, the homomorphisms $\pi_{\ell_\pm}^\Q(F) = \Q \to \pi_{\ell_\pm}^\Q(L)$ must be non-trivial, as desired.
\end{proof}

As mentioned in the introduction, it was shown in \cite{GKS1} that every homotopy $7$-sphere admits a metric of non-negative curvature.  Crucial to establishing this fact was the observation that every homotopy $7$-sphere admits a double disk-bundle decomposition with $\ell_\pm = 1$.  The following corollary shows that a similar strategy to obtain non-negative curvature on exotic spheres in higher dimensions will not work.


\begin{cor}
\label{C:spheres}
Suppose $M$ is a homotopy sphere.  Then $M$ admits a double disk-bundle decomposition $DB_- \cup_L DB_+$ with $B_\pm$ connected and $\codim(B_\pm) = 2$ if and only if it is either diffeomorphic to $\sph^n$, $n \in \{2,3,4,5\}$, or homeomorphic to $\sph^7$.
\end{cor}


\begin{proof}
If $M$ admits a double disk-bundle decomposition $DB_- \cup_L DB_+$ with $B_\pm$ connected and $\codim(B_\pm) = 2$, that is, with $\ell_\pm = 1$, then Table \ref{table:Qtype} implies that the corresponding homotopy fiber $F$ of the inclusion $L \to M$ has a factor $\Omega \sph^k$ with $k \in \{3,5,7\}$.  In the notation above, it follows that $s \in \{2,4,6\}$ and, by Theorem \ref{T:NoRank}, that $M$ has a homotopy group of positive rank in one of degrees $3$, $5$ or $7$.  As $M$ is a homotopy sphere, it must therefore be of dimension $n \in \{2, 3, 4, 5, 7\}$.   If $2 \leq n \leq 5$,  then $M$ must be diffeomorphic to $\sph^n$.  Indeed, if $n \in \{2,3,5\}$, this follows from the corresponding Smooth Poincar\'e Conjecture, while for $n = 4$, it was established in \cite{GT}.    

On the other hand, the standard actions of $\sph^1$ on $\sph^2$ and $T^2$  on $\sph^3$ are of cohomogeneity one and have 
codimension-two singular orbits.  Similar actions on $\sph^4$ and $\sph^5$ can be found in \cite{Par} and \cite{GVWZ}, respectively.  Finally, 
if $n = 7$, the construction in \cite{GKS1} ensures that every homotopy $7$-sphere admits a double disk-bundle decomposition with $\ell_\pm = 1$.
\end{proof}

Observe in Corollary \ref{C:spheres} that one still obtains the restriction $\dim(M) \in \{2,3,4,5,7\}$ under the weaker hypothesis that $M$ is only a simply connected \emph{rational} homotopy sphere; see \cite{Fa} for this and related observations.  However, in this case much remains  
unknown about which  
such $M$ admit a double disk-bundle decomposition with $\ell_\pm = 1$.  Indeed, although it was demonstrated in \cite{GKS1} that a large family of $2$-connected, rational $7$-spheres admit such a structure, it was subsequently shown in \cite{GKS2} that this family does not contain all possible homotopy types of such manifolds: for example, the family does not contain any $2$-connected $7$-manifold $M^7$ with $H^4(M^7) = \Z_5$ and non-standard linking form.  
In a forthcoming work, the authors will detail general obstructions to the existence of any double disk-bundle decomposition for highly connected, rational homology spheres, as well as for more general spaces.


\section{Double disk bundles in dimension at most five}
\label{S:dim5}

As discussed in the introduction, the smooth classification of simply connected manifolds of dimension at most four admitting a double disk-bundle decomposition is well known.  Using this, the case of manifolds of dimension at most five in Theorem \ref{T:thmA} is then a simple consequence of the preliminary results obtained in Section \ref{S:top}.


\begin{thm}
\label{T:Pf_Thm_A}
	Let $M$ be a smooth, closed, simply connected manifold which admits a double disk-bundle decomposition.  If $\dim(M)\leq 5$, then $M$ is rationally elliptic.  If $\dim(M) \leq 4$, then $M$ is diffeomorphic to one of $\sph^2$, $\sph^3$, $\sph^4$, $\CP^2$, $\sph^2 \x \sph^2$ or $ \CP^2 \# \pm \CP^2$.
\end{thm}


\begin{proof}
Since $\dim(M) \leq 5$, Propositions \ref{P:connected} and \ref{P:exceptional} imply that $M$ admits a decomposition $DB_-\cup_L DB_+$ with $B_\pm$ connected and $\dim(B_\pm) \leq 3$.  By Proposition \ref{P:L_BOUND}, it now follows that $M$ is rationally elliptic.

In dimensions $2$ and $3$, a smooth, closed, simply connected manifold must be diffeomorphic to a sphere, whereas the $4$-dimensional statement was proven by Ge and Radeschi \cite{GeRa}.
\end{proof}

In order to prove Theorem \ref{T:DIFFEO_CLASSIF}, assume for the remainder of this section that $M^5$ is a smooth, closed, simply connected $5$-manifold which admits a double disk-bundle decomposition $DB_-\cup_L DB_+$ with $B_\pm$ connected.  
It was already shown in Theorem \ref{T:Pf_Thm_A} that such an $M^5$ must be rationally elliptic.  Recall that, according to Pavlov \cite{Pav}, a five-dimensional, rationally elliptic manifold is rationally homotopy equivalent to either $\sph^5$ or $\sph^3\x \sph^2$. The following lemma will be helpful later.


\begin{lemma}
\label{L:COH_ISO}
Under the identifications \[
	H^\ast(\sph^2\times\sph^2)	 \cong \Z[x,y]/\{x^2 = y^2 = 0\} \ \text{ and } \
	H^\ast(\CP^2\# \overline{\CP}^2) \cong \Z[u,v]/\{u^2 + v^2 = uv = 0\},
\]
every automorphism of $H^\ast(\sph^2\times\sph^2)$ maps $\{\pm x,\pm y\}$ to itself,
while 
every automorphism of $H^\ast(\CP^2\# \overline{\CP}^2)$ maps $\{\pm(u+v), \pm(u-v)\}$ to itself.
\end{lemma}


\begin{proof}
In the case of $H^\ast(\sph^2\times \sph^2)$, notice that, for any $ax + by\in H^2(\sph^2\times \sph^2)$, the identity $(ax + by)^2 = 2ab\,xy$ holds.  Thus,  $(ax + by)^2 = 0$ if and only if either $a=0$ or $b=0$.  On the other hand, in the case of $H^\ast(\CP^2\# \overline{\CP}^2)$, $(au+bv)^2 = 0$ if and only if $a = \pm b$.  It follows that the sets $\{\pm x,\pm y\}$ and $\{\pm(u+v), \pm(u-v)\}$ characterize all primitive elements of degree two in their respective rings which square to $0$.  Therefore, these two sets are fixed by any automorphism of their respective cohomology rings.
\end{proof}

Theorem \ref{T:DIFFEO_CLASSIF} can now be proven by considering the two possible rational homotopy types of a rationally elliptic $5$-manifold separately.


\begin{theorem}  
\label{T:S5}
Suppose $M^5$ is a smooth, closed, simply connected $5$-manifold which is rationally homotopy equivalent to $\sph^5$.  If $M^5$ admits a double disk-bundle decomposition, then $M^5$ is diffeomorphic to either $\sph^5$ or the Wu manifold $\SU(3)/\SO(3)$.
\end{theorem}


\begin{proof}
By the Barden--Smale classification of smooth, closed, simply connected $5$-manifolds \cite{Ba,Sm3}, it suffices to show that $H_2(M^5)$ is either trivial or $\Z_2$.  By the Hurewicz Theorem and Poincar\'e duality, this is equivalent to establishing the same for either $\pi_2(M^5)$ or $H^3(M^5)$.  As $M^5 \simeq_\Q \sph^5$, by hypothesis, it is already clear that $\pi_2(M^5)$ and $H^3(M^5)$ are at most torsion.

Let $DB_- \cup_L DB_+$ be a double disk-bundle decomposition of $M^5$, with $B_\pm$ connected, and let $F$ be the homotopy fiber of the inclusion $L \to M^5$.  Since $M^5$ is rationally homotopy equivalent to $\sph^5$, Theorem \ref{T:NoRank} implies that the loop-space factor of 
$F$ must have non-trivial rational $\pi_4$ and, hence, the loop-space factor must be $\Omega \sph^5$.  It follows from Table \ref{table:Qtype} that $\pi_1(F) \neq \Z^2$ and, therefore, that $\rk(\pi_1(F)) \leq 1$.  Thus, by Lemma \ref{L:pi1B}, a singular leaf $B \in \{B_\pm\}$ must have finite fundamental group.  By Proposition \ref{P:exceptional}, $\dim(B)\leq 3$ and, therefore, $B$ is finitely covered by either $\sph^2$ or $\sph^3$.  In particular, this implies that $\pi_2(B)$ is either trivial or isomorphic to $\Z$.  The long exact homotopy sequence for the fibration $\sph^\ell \to L \to B$, where $1 \leq \ell \leq 4$, now yields that $\pi_2(L)$ is free abelian.

Consider the long exact homotopy sequence for the homotopy fibration $F \to L \to M^5$.  Assume first that $\pi_2^\Q(F) = 0$, that is, that $\pi_2(F)$ is at most torsion, and suppose that $\pi_2(M^5) \neq 0$, that is, that $M^5 \not\cong \sph^5$.  Since $\pi_2(L)$ is free abelian, it follows that $\pi_2(L) = 0$ and, hence, that the torsion group $\pi_2(M^5)$ injects into $\pi_1(F)$.  However, by Table \ref{table:Qtype}, $\pi_1(F)$ contains a torsion subgroup only if $\pi_1(F) \in \{Q_8, \Z \oplus \Z_2\}$.  Since the loop-space factor of 
$F$ is $\Omega \sph^5$, Table \ref{table:Qtype} implies that $\pi_1(F) = Q_8$ is impossible.  Therefore, $\pi_1(F) = \Z \oplus \Z_2$ and the only possibility is that $\pi_2(M^5) = \Z_2$, that is, that $M^5$ is diffeomorphic to the Wu manifold $\SU(3)/\SO(3)$.

Assume, on the other hand, that $\pi_2^\Q(F) \neq 0$.  As the loop-space factor of 
$F$ is $\Omega \sph^5$, it follows from Table \ref{table:Qtype} that $\pi_1(F) = 0$, $F \simeq_\Q \sph^2 \x \sph^2 \x \Omega \sph^5$ and $\codim(B_\pm) = 3$.  The long exact homotopy sequences for $F \to L \to M^5$ and $\sph^2 \to L \to B_\pm$ yield, in addition, that $\pi_1(L) = \pi_1(B_\pm) = 0$.  Therefore, $B_\pm \cong \sph^2$ and $L$ is an $\sph^2$-bundle over $\sph^2$.  Thus, $L$ is diffeomorphic to either $\sph^2 \x \sph^2$ or the non-trivial bundle $\CP^2\#\ol{\CP}^2$.

Since $M^5$ is simply connected and has the same rational cohomology as $\sph^5$, it follows from the Universal Coefficient Theorem that $H^2(M^5) = 0$.  The Mayer--Vietoris sequence for the decomposition $DB_- \cup_L DB_+$ now yields the short exact sequence 
\[
0\lra H^2(B_-)\oplus H^2(B_+)\stackrel{\pi_-^* - \pi_+^*}{\lra} H^2(L) \lra H^3(M^5)\lra 0,
\]
where $\pi_\pm \colon L \to B_\pm$ are the sphere-bundle projection maps.  In particular, the homomorphism $\pi_-^* - \pi_+^* \colon H^2(B_-)\oplus H^2(B_+)\to H^2(L)$ must be injective. Therefore, the images $\pi_\pm^*(z_\pm) \in H^2(L)$ of generators $z_\pm \in H^2(B_\pm)$ cannot differ by a sign, since, otherwise, there is an $\ve \in \{\pm 1\}$ such that $(\pi_-^* - \pi_+^*)(z_-, \ve z_+) = 0$.

Now, from the Gysin sequences for $\sph^2 \to L \stackrel{\pi_\pm}{\lra} B_\pm$, the images $\pi_\pm^*(z_\pm) \in H^2(L)$ are primitive elements which must necessarily square to zero.  Therefore, by Lemma \ref{L:COH_ISO}, together with the above observation that $\pi_-^*(z_-) \neq \pm \pi_+^*(z_+)$, it follows that there exist $\ve_1, \ve_2 \in \{\pm 1\}$ such that $\{\pi_-^*(z_-), \pi_+^*(z_+)\}$ is equal (as a set) to either $\{\ve_1 x, \ve_2 y\}$ or $\{\ve_1 (u+v), \ve_2 (u-v)\}$, depending on whether the bundle $L$ is trivial or not.  In turn, this implies that the map $\pi_-^* - \pi_+^* \colon H^2(B_-)\oplus H^2(B_+)\to H^2(L)$
is surjective whenever $L \cong \sph^2 \x \sph^2$, and of index $2$ whenever $L \cong \CP^2 \# \ol \CP^2$.  Finally, as desired, this implies that either $H^3(M^5) = 0$ or $H^3(M^5) = \Z_2$.
\end{proof}

To complete the proof of Theorem \ref{T:DIFFEO_CLASSIF}, it remains only to deal with the case of manifolds rationally homotopy equivalent to $\sph^3 \x \sph^2$.


\begin{theorem}
\label{T:S3S2}  
Suppose $M^5$ is a smooth, closed, simply connected $5$-manifold which is rationally homotopy equivalent to $\sph^3 \x \sph^2$.  
If $M^5$ admits a double disk-bundle decomposition, then $M^5$ is diffeomorphic to $\sph^3\times \sph^2$ or to the unique non-trivial $\sph^3$-bundle over $\sph^2$.
\end{theorem}


\begin{proof}
As in the proof of Theorem~\ref{T:S5}, the proof appeals to the Barden--Smale classification of closed, simply connected $5$-manifolds \cite{Ba,Sm3}. In particular, it is sufficient to show that $\pi_2(M^5) \cong H_2(M^5) \cong H^3(M^5) \cong \Z$.

To begin, let $DB_- \cup_L DB_+$ be a double disk-bundle decomposition of $M^5$, with $B_\pm$ connected, and let $F$ be the homotopy fiber of the inclusion $L \to M^5$.  By hypothesis, $\pi_2^\Q (M^5) = \Q$ and $\pi_3^\Q (M^5) = \Q^2$.  As $\pi_3^\Q (M^5)$
is the only odd-degree rational homotopy group of $M^5$ which is non-trivial, it follows from Theorem~\ref{T:NoRank} that the loop-space factor of 
$F$ must have non-trivial rational $\pi_2$.  Thus, by Table \ref{table:Qtype}, 
$F$ is rationally homotopy equivalent to one of $\sph^1\times \sph^1\times \Omega\sph^3$ or $\sph^2\times \Omega \sph^3$.

Suppose that 
$F$ is rationally homotopy equivalent to $\sph^1\times \sph^1\times \Omega\sph^3$.  From Table \ref{table:Qtype}, this implies that $\pi_1(F) = \Z^2$, that $\pi_2^\Q(F) = \Q$, and that $L$ is an $\sph^1$-bundle over each of the closed $3$-manifolds $B_\pm$.  From the long exact homotopy sequences for $F \to L \to M^5$ and $\sph^1 \to L \to B_\pm$, it is now apparent that $\pi_1 (L)$ and $\pi_1(B_\pm)$ are abelian.  As in the proof of Proposition \ref{P:L_BOUND}, this implies that each of $B_\pm$ is finitely covered by one of $\sph^3$, $\sph^2 \x \sph^1$ or $T^3$.  In particular, it follows that $\pi_2(B_\pm)$ is free abelian.  Therefore, from the long exact homotopy sequence for $\sph^1 \to L \to B_\pm$, it is now clear that $\pi_2(L)$ must also be free abelian.  However, given that $\pi_1(F) = \Z^2$, applying this fact to the long exact homotopy sequence for $F \to L \to M^5$ yields that $\pi_2(M^5)$ is free abelian.  Since $\rk(\pi_2(M^5)) = 1$, it may be concluded that $\pi_2(M^5) = \Z$, as desired.

Assume now that 
$F$ is rationally homotopy equivalent to
$\sph^2\times \Omega \sph^3$.  
From Table \ref{table:Qtype}, this implies that $\pi_1(F) = 0$ and that $L$ is an $\sph^2$-bundle over each of the closed $2$-manifolds $B_\pm$.   The long exact homotopy sequences for $F \to L \to M^5$ and $\sph^1 \to L \to B_\pm$ yield that $\pi_1(L) = \pi_1(B_\pm) = 0$.  Therefore, $B_\pm \cong \sph^2$ and, hence, $L$ is diffeomorphic to either $\sph^2 \x \sph^2$ or $\CP^2 \# \ol \CP^2$.

If $z_\pm \in H^2(B_\pm)$ are generators, then the Gysin sequences for $\sph^2 \to L \stackrel{\pi_\pm}{\lra} B_\pm$ yield that their images $\pi_\pm^*(z_\pm) \in H^2(L)$ are primitive elements which must necessarily square to zero. Therefore, by Lemma \ref{L:COH_ISO}, $\pi_\pm^*(z_\pm)$ lie in either $\{\pm x, \pm y\}$ or $\{\pm(u+v), \pm(u-v)\}$, depending on whether the bundle $L$ is trivial or not.

Now, since $M^5$ is simply connected and has the same rational cohomology as $\sph^3 \x \sph^2$, the Universal Coefficient Theorem yields $H^2(M^5) = \Z$.  The Mayer--Vietoris sequence for the decomposition $DB_- \cup_L DB_+$ therefore provides the exact sequence 
\[
0 \lra H^2(M^5) \lra H^2(B_-)\oplus H^2(B_+)\stackrel{\pi_-^* - \pi_+^*}{\lra} H^2(L) \lra H^3(M^5) \lra 0.
\]
In particular, the homomorphism $\pi_-^* - \pi_+^*\colon H^2(B_-)\oplus H^2(B_+)\to H^2(L)$ has kernel isomorphic to $H^2(M^5) = \Z$ and, hence, some linear combination of $\pi_\pm^*(z_\pm)$ must be trivial.  However, by the above observations about these elements, this is impossible unless $\pi_\pm^*(z_\pm)$ agree up to sign.  Therefore, in either case,
\[
H^3(M^5) \cong H^2(L) / \<\pi_-^*(z_-), \pi_+^*(z_+)\> \cong H^2(L) / \<\pi_-^*(z_-)\> \cong \Z
\]
as desired.
\end{proof}


\section{Double disk bundles in dimension \texorpdfstring{$6$}{6}}
\label{S:DIM_6}

As mentioned in the introduction, there exist closed, simply connected smooth $6$-di\-men\-sional counter-examples to the rational ellipticity of double disk bundles.  The main goals of this section are to establish Theorem \ref{T:thmA} in dimension six and Theorem \ref{T:6dim_b2}.

Throughout this section, $M^6$ will denote a smooth, closed, simply connected $6$-manifold which admits a double disk-bundle decomposition $DB_- \cup_L DB_+$ with $B_\pm$ connected.

As a consequence of Proposition \ref{P:L_BOUND}, $M^6$ is rationally elliptic whenever one of $B_\pm$ is of codimension $\geq 3$.  Therefore, only the case of codimension-two singular leaves $B_\pm$ needs to be considered in what follows.  In this case, Table \ref{table:Qtype} yields that the homotopy fiber $F$ of the inclusion $L \to M^6$ has $\pi_1(F) \in \{Q_8, \Z \oplus \Z_2, \Z^2\}$.  


\begin{lem}
\label{L:fundgpF}
Suppose that the singular leaves $B_\pm$ in the double disk-bundle decomposition of $M^6$ are both of codimension two and that some $B \in \{B_\pm\}$ has finite fundamental group.  Then
$$
\rank(\pi_1(F)) + \rank(\pi_2(B)) = b_2(M^6) + 1 \,.
$$ 
\end{lem}


\begin{proof}
Since the fundamental group $\pi_1(B)$ is finite, the universal cover $\tilde B$ of $B$ is a smooth, closed, simply connected $4$-manifold which satisfies Poincar\'e duality and has $\pi_j(\tilde B) = \pi_j(B)$ for all $j \geq 2$.  Together with the Hurewicz and Universal Coefficient Theorems, it may thus be concluded that $H_2( \tilde B) \cong \pi_2(\tilde B) = \pi_2(B)$ is free abelian.

From the long exact homotopy sequence for the fibration $\sph^1 \to L \to B$, it follows that $\pi_2(L)$ is also free abelian and
\beq
\label{E:LES1}
\rank(\pi_2(B)) = \rank(\pi_2(L)) - \rank(\pi_1(L)) + 1 \,.
\eeq

By Table~\ref{table:Qtype}, there is a unique $j_0 \in \N$, such that $\pi_{2j_0}(F)$ is of positive rank ($=1$).  If $\pi_1(F) \in \{Q_8, \Z \oplus \Z_2\}$, then $\rank(\pi_2(F)) = 0$, while $\rank(\pi_2(F)) = 1$ whenever $\pi_1(F) = \Z^2$.  In this latter case, Table~\ref{table:Qtype} yields that 
$F$ is rationally homotopy equivalent to $\sph^1 \x \sph^1 \x \Omega \sph^3$. Therefore, as a consequence of Theorem~\ref{T:NoRank}, the image of the homomorphism $\pi_{2}(F) \to \pi_{2}(L)$ in the long exact sequence for the homotopy fibration $F \to L \to M^6$ has rank $0$.  In other words, the free abelian group $\pi_2(L)$ must inject into $\pi_2(M^6)$.  

It now follows from exactness and the Hurewicz Theorem that, in all three cases, 
\beq
\label{E:LES2}
\rank(\pi_1(F)) = b_2(M^6) - \rank(\pi_2(L)) + \rank(\pi_1(L)) \,.
\eeq
Equations \eqref{E:LES1} and \eqref{E:LES2} together yield $\rank(\pi_1(F)) + \rank(\pi_2(B)) = b_2(M^6) + 1$, as desired.
\end{proof}


\begin{thm}
\label{T:nonoriented}
Suppose that the singular leaves $B_\pm$ in the double disk-bundle decomposition of $M^6$ are both of codimension two 
and that at least one of the bundles $\sph^1 \to L \to B_\pm$ is non-orientable.  Then $M^6$ is rationally elliptic with $b_2(M^6) \leq 2$.
\end{thm}


\begin{proof}
Observe first that, by Table \ref{table:Qtype}, the hypothesis that at least one of the bundles $\sph^1 \to L \to B_\pm$ is non-orientable is equivalent to $\pi_1(F)$ being either $Q_8$ or $\Z \oplus \Z_2$.  Therefore, by Lemma \ref{L:pi1B}, some $B \in \{B_\pm\}$ has finite fundamental group.  By Lemma \ref{L:fundgpF}, if $b_2(M^6) - \rank(\pi_1(F)) \leq 1$, then $\rank(\pi_2(B)) \leq 2$.  In this case, the Hurewicz and Universal Coefficient Theorems ensure that the universal cover $\tilde B$ of $B$ has $H_2( \tilde B) = \pi_2(\tilde B) = \pi_2(B)$ free abelian of rank at most $2$.  From the classification of smooth, closed, simply connected $4$-manifolds \cite{Fr}, it now follows that $\tilde B$ is homeomorphic to one of $\sph^4$, $\CP^2$, $\sph^2 \x \sph^2$ or $\CP^2 \# \pm \CP^2$ and, hence, rationally elliptic.  As $\pi_j(\tilde B) = \pi_j(B)$ for all $j \geq 2$, the rational ellipticity of $M^6$ now follows from Lemma \ref{L:high_enough}.

It remains, therefore, to show that the inequality $b_2(M^6) - \rank(\pi_1(F)) > 1$ is not possible under the present hypotheses.  To this end, note that the integral homology of $F$ has been determined in Table 1.5 of \cite{GH}.  It is a simple application of the Universal Coefficient Theorem to compute the rational cohomology groups of $F$, namely,
\begin{align}
\label{E:FQ8}
H^j(F; \Q) &= 
\begin{cases}
\Q, & j = 0, \\
\Q^2, & j >0 \text{ and } j \equiv 0 \!\!\!\mod 3, \\
0, & \text{otherwise}
\end{cases}
\qquad \text{ if } \pi_1(F) = Q_8 \,,
\\
H^j(F; \Q) &= 
\begin{cases}
\Q, & j = 0 \text{ or } j \text{ odd}, \\
\Q^2, & j >0 \text{ and } j \equiv 0 \!\!\!\mod 4, \\
0, & \text{otherwise}
\end{cases} 
\qquad \text{ if } \pi_1(F) = \Z \oplus \Z_2 \,.
\end{align}

Consider the rational Serre spectral sequence $(E_j,d_j)$ associated to the homotopy fibration $F \to L \to M^6$.   In particular, $H^5(L; \Q) = \bigoplus_{k+l = 5} E_\infty^{k,l}$.  On the other hand, as $L$ is a codimension-$1$ submanifold of the closed, simply connected $6$-manifold $M^6$, it is orientable and, hence, has $H^5(L; \Q) = \Q$; see, for example, \cite[p.~107]{Hi}.  These facts will place restrictions on the Betti numbers of $M^6$.  For convenience, denote by $d_j^{k,l}$ the differential $d_j \colon E_j^{k,l} \to E_j^{k+j, l+1-j}$ and by $\Delta_j^5$ the diagonal $\{E_j^{k,l} \mid k + l = 5\}$ on the $E_j$-page of $(E_j,d_j)$.

If $\pi_1(F) = Q_8$, the only non-trivial entry on the diagonal $\Delta_2^5$ is 
\[
E_2^{2,3} = H^2(M^6; H^3(F; \Q)) = H^2(M^6; \Q^2) = \Q^{2 \, b_2(M^6)} .
\]
As $d_4^{2,3} = (d_4\colon E_4^{2,3} = E_2^{2,3} \to E_4^{6,0} = \Q)$ is the only possible non-trivial differential on any page which involves $E_2^{2,3}$, it follows that $\Q = H^5(L; \Q) = \ker(d_4^{2,3}) \In \Q^{2 \, b_2(M^6)}$.  However, since $\rank(d_4^{2,3}) \leq 1$, one concludes that 
\[
2 \, b_2(M^6) - 1 \leq 2 \, b_2(M^6) - \rank(d_4^{2,3}) = \dim(\ker(d_4^{2,3})) = 1 \leq 2 \, b_2(M^6) ,
\]
which immediately yields $b_2(M^6) - \rank(\pi_1(F)) = b_2(M^6) = 1$.

Suppose now that $\pi_1(F) = \Z \oplus \Z_2$.  The non-trivial entries on the diagonal $\Delta_2^5$ consist of $E_2^{0,5} = \Q$, $E_2^{2,3} = \Q^{b_2(M^6)}$ and $E_2^{4,1} = \Q^{b_2(M^6)}$.  By considering all possible differentials which have these entries as either domain or range, one obtains that the contribution of $E_2^{0,5}$ to $H^5(L; \Q) = \Q$ has rank 
\[
1 - \rank(d_2^{0,5}) - \rank(d_3^{0,5}) - \rank(d_6^{0,5}), 
\]
while the contribution of $E_2^{2,3}$ has rank
\[
b_2(M^6) - \rank(d_2^{0,4}) - \rank(d_4^{2,3})
\]
and $E_2^{4,1}$ contributes rank
\[
b_2(M^6) - \rank(d_2^{4,1}) - \rank(d_4^{0,4}).
\]
Therefore, the rank of $H^5(L; \Q) = \Q$ is given by
\begin{align*}
1 &= (1 - \rank(d_2^{0,5}) - \rank(d_3^{0,5}) - \rank(d_6^{0,5})) \\
&\hspace*{15mm} + (b_2(M^6) - \rank(d_2^{0,4}) - \rank(d_4^{2,3})) \\
&\hspace*{15mm} + (b_2(M^6) - \rank(d_2^{4,1}) - \rank(d_4^{0,4}))  \\
&= 1 + 2\, b_2(M^6) - (\rank(d_2^{0,5}) + \rank(d_3^{0,5}) + \rank(d_6^{0,5})) \\
&\hspace*{15mm} - (\rank(d_2^{0,4}) + \rank(d_4^{0,4})) \\
&\hspace*{15mm} - (\rank(d_2^{4,1}) + \rank(d_4^{2,3})) .
\end{align*}
Now, since $E_2^{0,4} = \Q^2$, $E_2^{0,5} = \Q$ and $E_2^{6,0} = \Q$, this implies
\[
1 \geq 2 \, b_2(M^6) + 1 - 1 - 2 - 1 = 2 \, b_2(M^6) - 3,
\]
which immediately yields $b_2(M^6) \leq 2$, hence, $b_2(M^6) - \rank(\pi_1(F)) \leq 1$.
\end{proof}

As a consequence of Theorem \ref{T:nonoriented}, one can make the following additional general observation, which may be useful in its own right.


\begin{cor}
\label{C:b2big}
Suppose that the singular leaves $B_\pm$ in the double disk-bundle decomposition of $M^6$ are both of codimension two and that $b_2(M^6) \geq 3$.  Then both of the bundles $\sph^1 \to L \to B_\pm$ are orientable.
\end{cor}

In the case where each singular leaf is of codimension two and has infinite fundamental group, a classification up to diffeomorphism is achieved.  In particular, decomposing $\sph^3 \x \sph^3$ via the well-known decomposition of one factor into two solid tori is precisely of this form, having singular leaves diffeomorphic to $\sph^1 \x \sph^3$.


\begin{thm}
\label{T:infinitePi1}
Suppose that the singular leaves $B_\pm$ in the double disk-bundle decomposition of $M^6$ are both of codimension two, with $\rank(\pi_1(B_\pm)) \geq 1$.  Then $M^6$ is diffeomorphic to $\sph^3 \x \sph^3$ and, hence, rationally elliptic.
\end{thm}


\begin{proof}
Since $B_\pm$ have infinite fundamental groups, Lemma \ref{L:pi1B}, together with Table~\ref{table:Qtype}, implies that $\pi_1(F) = \pi_1(L) = \Z^2$.  Furthermore, recall that, since $\pi_1(M^6) = 0$, equation~(3.7) of \cite{GH} implies that $\pi_1(L)$ is generated by the images of the homomorphisms $\pi_1(\sph^1) \to \pi_1(L)$ in the long exact homotopy sequences for the bundles $\sph^1 \to L \to B_\pm$. 
Therefore, $\pi_1(B_\pm) = \Z$ and $\pi_2(L) = \pi_2(B_\pm)$.  By applying the Hurewicz and Universal Coefficient Theorems, it follows in addition that $H^2(B_\pm)$ is free abelian.

By Table \ref{table:Qtype}, the circle bundles $\sph^1 \to L \to B_\pm$ are orientable.  Thus, by Proposition 6.15 of \cite{Mo}, these are principal $\sph^1$-bundles and, therefore, are determined by their Euler classes $e_\pm \in H^2(B_\pm)$.  Moreover, as $M^6$ is simply connected, the regular leaf $L$ is also orientable.  
Altogether, this implies that the $4$-manifolds $B_\pm$ are orientable and, in particular, satisfy Poincar\'e duality. 

By Theorem \ref{T:triv}, the bundles $\sph^1 \to L \to B_\pm$ are trivial, that is, $L \cong \sph^1 \x B_\pm$ and $e_\pm = 0$.  The respective Gysin sequences then yield, in addition, that $H^2(L ) = \Z^{b_2(B_\pm) + 1}$; in particular, this implies that $b_2(B_-) = b_2(B_+)$.

By Proposition \ref{P:PDmiddle}, the maximal free abelian cover $\ol L$ of $L$ satisfies $H^*(\ol L; \Q) \cong H^*(\sph^3; \Q)$.  Since $\pi_2(L) = \pi_2(B_\pm)$, it thus follows from the rational Hurewicz Theorem that 
$$
\rk(\pi_2(B_\pm)) = \rk(\pi_2(L)) = \rk(\pi_2(\ol L)) = b_2(\ol L) = 0 \,.
$$ 
Therefore, by Lemma \ref{L:nilp}, $B_\pm$ are nilpotent spaces with $\pi_1^\Q(B_\pm) = \Q$ and $\pi_2^\Q(B_\pm) = 0$.  From their minimal models, it follows that $b_2(B_\pm) = 0$ and, being free abelian, that $H^2(B_\pm) = 0$.

The Mayer--Vietoris sequence corresponding to the double disk-bundle decomposition of $M^6$ now yields 
\[
0 \to H^1(B_-) \oplus H^1(B_+) = \Z^2  \to H^1(L) = \Z^2 \to H^2(M^6) = \Z^{b_2(M^6)} \to 0 \,.
\]
Since $H^2(M^6)$ is free abelian, the injection $H^1(B_-) \oplus H^1(B_+) = \Z^2 \to H^1(L) = \Z^2$ is an isomorphism, from which one concludes that $H^2(M^6) = 0$.  Furthermore, from
\[
0 \to H^2(L) = \Z \to H^3(M^6) \to H^3(B_-) \oplus H^3(B_+) = \Z^2 \to \cdots 
\]
it is clear that $H^3(M^6)$ is free abelian of rank $0 < b_3(M^6) \leq 3$.  However, being a $6$-manifold, $b_3(M^6)$ must be even.  Therefore, $H^3(M^6) = \Z^2$ and, from the classification of closed, simply connected smooth $6$-manifolds \cite{Ju, Wa, Zh}, it follows that $M^6$ is diffeomorphic to $\sph^3 \x \sph^3$.
\end{proof}

All the ingredients necessary to prove Theorem~\ref{T:thmA} in dimension six are now in place.


\begin{thm}
\label{T:dim6}
Let $M^6$ be a closed, smooth, simply connected $6$-manifold with second Betti number $b_2(M^6) \leq 3$ which admits a double disk-bundle decomposition.
Then $M^6$ is rationally elliptic.
\end{thm}


\begin{proof}
By Proposition \ref{P:L_BOUND}, it suffices to consider singular leaves $B_\pm$ of codimension two.  Suppose that there is some $B \in \{B_\pm\}$ with finite fundamental group.  If both of the bundles $\sph^1 \to L \to B_\pm$ are orientable, that is, if $\pi_1(F) = \Z^2$, then, by Lemma~\ref{L:fundgpF} together with the hypothesis $b_2(M^6) \leq 3$, it is clear that $\rank(\pi_2(B)) = b_2(M^6) - 1 \leq 2$.  As in the proof of Theorem~\ref{T:nonoriented}, it follows that $M^6$ is rationally elliptic.  All remaining cases have been dealt with in Theorems~\ref{T:nonoriented} and \ref{T:infinitePi1}, thus completing the proof.
\end{proof}


\begin{remark}
\label{r:6hyp}
Notice that the hypothesis $b_2(M^6) \leq 3$ has been used in only one scenario, namely, in the case where the singular leaves $B_\pm$ are of codimension $2$, at least one of $\pi_1(B_\pm)$ is finite and the bundles $\sph^1 \to L \to B_\pm$ are both orientable, that is, $\pi_1(F) = \Z^2$ by Table \ref{table:Qtype}.  In all other cases, assuming only that $M^6$ admits a double disk-bundle decomposition ensures that $M^6$ is rationally elliptic.
\end{remark}

In light of Theorem~\ref{T:infinitePi1}, it is tempting to seek a classification up to diffeomorphism of rationally elliptic $6$-manifolds which admit a double disk-bundle decomposition.  However, as suggested by the work in \cite{He}, such a classification seems out of reach at present.  Nevertheless, imposing further restrictions on the Betti numbers allows one to make some progress. 


\begin{thm}
\label{T:smoothS6}
Let $M^6$ be a closed, smooth, simply connected $6$-manifold with $H^*(M^6; \Q) = H^*(\sph^6; \Q)$  which admits a double disk-bundle decomposition. Then $M^6$ is diffeomorphic to $\sph^6$.
\end{thm}


\begin{proof}
By Smale's 
resolution of the Generalized Poinc\'are Conjecture \cite{Sm2}, it suffices to show that $M^6$ is an integral (co)homology sphere.  Since $M^6$ is simply connected, it is clear from the Universal Coefficient Theorem and Poincar\'e duality that $H^j(M^6) = 0$, for $j \in \{1,2,5\}$, and that $H^3(M^6) \cong H^4(M^6)$.  Therefore, it is enough to show that the torsion group $H^3(M^6) \cong H^4(M^6)$ is trivial.

By Theorem~\ref{T:dim6}, $M^6$ is rationally elliptic.  Thus, by the rational Hurewicz Theorem and the relations \eqref{E:degrels}, the only non-trivial rational homotopy groups of $M^6$ are $\pi_6^\Q(M^6) = \Q$ and $\pi_{11}^\Q(M^6) = \Q$.  From Theorem \ref{T:NoRank}, it now follows that 
the homotopy fiber $F$ of the inclusion $L \to M^6$ has a loop-space factor $\Omega \sph^6$ or $\Omega \sph^{11}$.  From Table \ref{table:Qtype}, together with the fact that $1 \leq \ell_\pm \leq 5$, this implies that  
$F$ is rationally homotopy equivalent to $\sph^{\ell_-} \x \sph^{\ell_+} \x \Omega \sph^{\ell_- + \ell_+ + 1}$, with $\{\ell_\pm\} = \{1,4\}$ or $\{2,3\}$, or else to $\sph^{5} \x \sph^{5} \x \Omega \sph^{11} \simeq_\Q \sph^5 \x \Omega \sph^6$, 
with $\ell_\pm = 5$.

Observe, however, that $\ell_\pm = 5$ implies that the singular leaves $B_\pm$ are points and, hence, that $M^6$ is the union of two $6$-dimensional disks.  Consequently, in this case $M^6$ is homeomorphic and, thus, diffeomorphic to $\sph^6$.  

Suppose, on the other hand, that $\{\ell_\pm\} = \{1,4\}$.  From Table \ref{table:Qtype} it follows that $\pi_1(F) = \Z$, while the Hurewicz Theorem and Poincar\'e duality ensure that $\pi_2(M^6) \cong H_2(M^6) \cong H^4(M^6)$ is torsion.  Therefore, the homomorphism $\pi_2(L) \to \pi_2(M^6)$ in the long exact homotopy sequence for $F \to L \to M^6$ must be surjective.  On the other hand, since one of the singular leaves $B_\pm$ is a connected, codimension-five submanifold of $M^6$, the regular leaf $L$ is an $\sph^4$-bundle over $\sph^1$.  In particular, the long exact homotopy sequence for this bundle yields $\pi_2(L) = 0$ and, hence, $H^4(M^6) \cong \pi_2(M^6) = 0$, as desired.

Suppose, finally, that $\{\ell_\pm\} = \{2,3\}$.  By Table \ref{table:Qtype}, $\pi_1(F) = 0$ and the bundles $\sph^{\ell_\pm} \to L \to B_\pm$ are both orientable.  It follows from the long exact homotopy sequence for $F \to L \to M^6$ that $\pi_1(L) = 0$, while the long exact homotopy sequences for $\sph^{\ell_\pm} \to L \to B_\pm$ yield $\pi_1(B_\pm) = 0$.  Since $\{\ell_\pm\} = \{2,3\}$, the classification of surfaces and Perelman's resolution of the Poincar\'e Conjecture \cite{Pe1, Pe2, Pe3} imply $\{B_\pm\} = \{\sph^2, \sph^3\}$.  Hence, $L$ is the total space of orientable bundles $\sph^3 \to L \to \sph^2$ and $\sph^2 \to L \to \sph^3$.  In particular, from the Gysin sequence for $\sph^3 \to L \to \sph^2$, the regular leaf $L$ has the same cohomology as $\sph^3 \x \sph^2$ and the bundle projection induces an isomorphism $H^2(\sph^2) \to H^2(L)$.  Since $\{B_\pm\} = \{\sph^2, \sph^3\}$ and $H^3(M^6)$ is torsion, applying this observation to the Mayer--Vietoris sequence for the decomposition $DB_- \cup_L DB_+$ of $M^6$ now yields
\[
H^2(M^6) = 0 \to H^2(B_-) \oplus H^2(B_+) = \Z \stackrel{\cong}{\lra} H^2(L) = \Z \to H^3(M^6) \to 0,
\]
from which it follows that $H^3(M^6) = 0$, as desired.
\end{proof}

Recall that Wall's Splitting Theorem \cite{Wa} implies that every closed, smooth, simply connected $6$-manifold $M^6$ splits as a connected sum $M^6_0 \# M^6_1$, where $M^6_0$ has finite $H^3(M^6_0)$ and $M^6_1$ is a connected sum of $b_3(M^6)/2$ copies of $\sph^3 \x \sph^3$.  As a consequence of the following theorem, if such a manifold is rationally hyperbolic and admits a double disk-bundle decomposition, then $b_3(M^6) = 0$.


\begin{thm}
\label{T:boom}
Let $M^6$ be a closed, smooth, simply connected $6$-manifold with 
$b_3(M^6) \neq 0$ which admits a double disk-bundle decomposition.  Then $M^6$ is diffeomorphic to $\sph^3 \x \sph^3$.
\end{thm}

\begin{proof}
Suppose first that $M^6$ is rationally hyperbolic.  By Theorem \ref{T:dim6} and Remark \ref{r:6hyp}, this is possible only if the singular leaves $B_\pm$ are both of codimension two, at least one of $\pi_1(B_\pm)$ is finite, the bundles $\sph^1 \to L \to B_\pm$ are both orientable, and $\pi_1(F) = \Z^2$.  By \cite[p.~107]{Hi}, $L$ is orientable and, therefore, so too are $B_\pm$.  As such, $B_\pm$ both satisfy Poincar\'e duality and, hence, $b_1(B_\pm) = b_3(B_\pm)$.

Now, recall that, by excision and Poincar\'e-Lefschetz duality, there are isomorphisms $H^j(M^6, DB_\pm) \cong H^j(DB_\mp, L) \cong H_{6-j}(B_\mp)$ for all $j \geq 0$.  Therefore, from the portion 
$$
\dots \to H^3(M^6, DB_\pm) \to H^3(M^6) \to H^3(B_\pm) \to \cdots
$$
of the long exact sequence for the pair $(M^6, DB_\pm)$ it follows that
$$
b_3(M^6) \leq b_1(B_-) + b_1(B_+) \leq 1\,,
$$
where the final inequality follows from Lemma \ref{L:pi1B} and the long exact homotopy sequences for $\sph^1 \to L \to B_\pm$.  However, since $b_3(M^6)$ must be even, one concludes that $b_3(M^6) = 0$, a contradiction.

Assume, therefore, that $M^6$ is rationally elliptic.  Since $b_3(M^6) \neq 0$, by the work of Pavlov \cite{Pav} (see also \cite{He}), $M^6$ must be rationally homotopy equivalent to $\sph^3 \x \sph^3$.  In particular, $\rk(\pi_2(M^6)) = 0$ and the only non-trivial rational homotopy group is $\pi_3^\Q(M^6) = \Q^2$.   
By Theorem \ref{T:NoRank},  
the homotopy fiber $F$ of the inclusion $L \to M^6$ has a loop-space factor $\Omega \sph^2$ or $\Omega \sph^3$.  By Table \ref{table:Qtype}, only $\Omega \sph^3$ is possible and there are only two possible scenarios: either $\ell_\pm = 1$ and $\pi_1(F) = \Z^2$, or else $\ell_\pm = 2$, $\pi_1(F) = 0$ and $F \simeq_\Q \sph^2 \x \Omega \sph^3$.

Suppose $\ell_\pm = 1$ and $\pi_1(F) = \Z^2$.  Since $\rk(\pi_2(M^6)) = 0$, 
 the long exact homotopy sequence for $F \to L \to M^6$ yields 
 that $\pi_1(L) = \pi_1(F) = \Z^2$.  Applying equation (3.7) of \cite{GH} to the long exact homotopy sequences for $\sph^1 \to L \to B_\pm$ now yields that $\pi_1(B_\pm) = \Z$.  By Theorem \ref{T:infinitePi1}, it follows that $M^6$ is diffeomorphic to $\sph^3 \x \sph^3$.

Suppose, on the other hand, that $\ell_\pm = 2$, $\pi_1(F) = 0$ and $F \simeq_\Q \sph^2 \x \Omega \sph^3$.  By Table \ref{table:Qtype}, the bundles $\sph^2 \to L \to B_\pm$ are both orientable and it follows from the long exact homotopy sequences for $F \to L \to M^6$ and $\sph^{\ell_\pm} \to L \to B_\pm$ that $\pi_1(B_\pm) = \pi_1(L) = 0$.  Being closed, simply connected $3$-manifolds, it follows from Perelman's resolution of the Poincar\'e Conjecture \cite{Pe1, Pe2, Pe3} that $B_\pm = \sph^3$.  Therefore, by the Gysin sequence, $L$ has the integral cohomology of $\sph^3 \x \sph^2$.  The Mayer--Vietoris sequence for the decomposition $DB_- \cup_L DB_+$ of $M^6$ now gives
$$
H^2(B_-) \oplus H^2(B_+) = 0 \to H^2(L) = \Z \to H^3(M^6) \to H^2(B_-) \oplus H^2(B_+) = \Z^2 \to \cdots.
$$
Therefore, by exactness, $H^3(M^6)$ is free abelian, that is, $H^3(M^6) = \Z^2$.  Since $H^*(M^6; \Q) = H^*(\sph^3 \x \sph^3; \Q)$, it follows from the Universal Coefficient Theorem and Poincar\'e duality that $M^6$ is an integral cohomology $\sph^3 \x \sph^3$.  By the classification of closed, smooth, simply connected $6$-manifolds \cite{Ju,Wa,Zh}, it follows that $M^6$ is diffeomorphic to $\sph^3 \x \sph^3$.
\end{proof}

Theorems \ref{T:smoothS6} and \ref{T:boom} together give a characterization of all $6$-dimensional double disk bundles with vanishing second Betti number.


\begin{cor}
Let $M^6$ be a closed, smooth, simply connected $6$-manifold with $b_2(M^6) = 0$ which admits a double disk-bundle decomposition.  Then $M^6$ is diffeomorphic to either $\sph^6$ or $\sph^3 \x \sph^3$.
\end{cor}

\begin{proof}
By Poincar\'e duality, the only possible non-trivial Betti number is $b_3(M^6)$.  The result now follows easily from 
Theorems \ref{T:smoothS6} 
and \ref{T:boom}.
\end{proof}


\section{Double disk bundles in dimension \texorpdfstring{$7$}{7}}
\label{S:dim7}

In this section, 
Theorem~\ref{T:thmA} will be proven in dimension seven via a careful analysis of all possible cases.   
Throughout, $M^7$ will denote a smooth, closed, simply connected $7$-manifold which admits a double disk-bundle decomposition $DB_- \cup_L DB_+$ with $B_\pm$ connected.  As before, let $F$ denote the homotopy fiber of the inclusion $L \to M^7$.

By Proposition \ref{P:L_BOUND}, $M^7$ is rationally elliptic whenever one of $B_\pm$ is of codimension $\geq 4$.  Therefore, together with Proposition \ref{P:exceptional}, it may be assumed that the fibers of the bundles $\sph^{\ell_\pm} \to L \to B_\pm$ satisfy $1 \leq \ell_\pm \leq 2$.  


\begin{thm}
\label{P:rk0}
If the bundles $\sph^{\ell_\pm} \to L \to B_\pm$ are both non-orientable, then $M^7$ is rationally elliptic.
\end{thm}


\begin{proof}
By Table \ref{table:Qtype}, the hypothesis is equivalent to taking $\ell_\pm = 1$, $\pi_1(F) = Q_8$ and  
$F$ to be rationally homotopy equivalent to $\sph^3 \x \sph^3 \x \Omega \sph^7$.  From the long exact homotopy sequence for $F \to L \to M^7$, this implies, in particular, that $\pi_1(L)$ is finite.

Now, consider the Serre spectral sequence $(E_j, d_j)$ associated to $F\to L\to M^7$, where the rational cohomology of $F$ is given by \eqref{E:FQ8}.  In particular,  
no non-trivial differential can hit either $E_2^{5,0} = H^5(M^7;\Q) \cong \Q^{b_2(M)}$ or $E_2^{2,3} = H^2(M^7; \Q) \ox H^3(F; \Q) \cong \Q^{2 \, b_2(M^7)}$.  Thus, these entries survive to the $E_\infty$-page and, being the only non-trivial entries on the diagonal $\{E_2^{k,l} \mid k + l = 5\}$, it follows that  $H^5(L; \Q) = \bigoplus_{k+l = 5} E_\infty^{k,l} \cong \Q^{3\, b_2(M^7)}$.  On the other hand, as $L$ is a codimension-one submanifold of the closed, simply connected manifold $M^7$, it is orientable by \cite[p.~107]{Hi}.  Therefore, $L$ satisfies Poincar\'e duality and, since $\pi_1(L)$ is finite, it follows that $H^5(L; \Q) = 0$.  Hence, $b_2(M^7) = 0$.

Consequently, all entries on the diagonal $\{E_2^{k,l} \mid k + l = 2\}$ are trivial, which in turn implies that $H^2(L; \Q) = 0$.   By Poincar\'e duality, it now follows that $H^4(L;\Q) = 0$.  Therefore, in the spectral sequence $(E_j, d_j)$, the differential $d_4^{0,3} : E_4^{0,3} \cong H^3(F; \Q) = \Q^2 \to E_4^{4,0} \cong H^4(M^7; \Q) = \Q^{b_4(M^7)}$ must be surjective, so $b_4(M^7) = 2 - \dim(\ker d_4^{0,3})$.  On the other hand, the only non-trivial entries on the diagonal $\{E_5^{k,l} \mid k + l = 3\}$ are $E_5^{0,3} \cong \Q^{\dim(\ker d_4^{0,3})}$ and $E_5^{3,0} \cong H^3(M^7; \Q) = \Q^{b_3(M^7)}$, and both of these survive to the $E_\infty$-page.  Therefore, $H^3(L;\Q) = \Q^{b_3(M^7) + \dim(\ker d_4^{0,3})}$.  However, by Poincar\'e duality, $b_3(M^7) = b_4(M^7)$ and, hence, 
$$
b_3(L) = b_3(M^7) + \dim(\ker d_4^{0,3}) = b_4(M^7) + \dim(\ker d_4^{0,3}) = 2.
$$
Altogether, these observations imply that $H^*(L;\Q) = H^*(\sph^3 \x \sph^3; \Q)$.  Since $L$ is nilpotent, by Theorem 1.3 of \cite{GH}, and $\pi_1(L)$ is finite, it has a (simply connected) minimal model.  Moreover, 
since a product of spheres is intrinsically formal, it follows that $L \simeq_\Q \sph^3 \x \sph^3$.    
This implies, in particular, that $\pi_j^\Q(L) 
= 0$ for all $j \geq 4$.  By Lemma \ref{L:high_enough}, it now follows that $M^7$ is rationally elliptic.
\end{proof}


\begin{remark}  
Notice that in Theorem \ref{P:rk0} there were no restrictions placed on $b_2(M^7)$.  Together with Theorem \ref{T:nonoriented} and Theorem \ref{T:Pf_Thm_A}, it follows that a closed, smooth, simply connected manifold of dimension $\leq 7$ that admits a double disk-bundle decomposition $DB_- \cup_L DB_+$ for which both of the bundles $\sph^{\ell_\pm} \to L \to B_\pm$ (equivalently, both of $B_\pm$) are non-orientable must be rationally elliptic.  This result is false in all dimensions $\geq 8$. 
To see this, first observe that $\sph^4$ admits a well-known $\SO(3)$ action of cohomogeneity one, with singular orbits diffeomorphic to $\RP^2$.  Furthermore, in every dimension $\geq 4$, there exist infinitely many closed, smooth, simply connected, rationally hyperbolic manifolds.  If $N$ is one such manifold, it now follows from Proposition \ref{P:SuffCond}\eqref{i:base} that the closed, smooth, simply connected, rationally hyperbolic manifold $\sph^4 \x N$ admits a double disk-bundle decomposition with non-orientable singular leaves diffeomorphic to $\RP^2 \x N$.
\end{remark}

As a consequence of the standard decomposition of $\sph^3$ into a union of two solid tori, for every closed, simply connected, smooth $4$-manifold $X^4$ there is a double disk-bundle decomposition induced on the product $M^7 = X^4 \x \sph^3$ such that the bundles $\sph^1 \to L \to B_\pm$ are both orientable.  
In particular, if $X^4$ is rationally hyperbolic, then so too is $M^7$.  Moreover, as $b_2(\#_{k = 1}^n \CP^2) = n$, there are rationally hyperbolic manifolds $M^7 = X^4 \x \sph^3$ achieving every possible $b_2 (M^7) \geq 3$.  If one is interested in rational ellipticity in the case that both circle bundles are orientable, it turns out that $b_2(M^7)$ is the only obstruction.


\begin{thm}
\label{P:rk2}	
Suppose that  
the singular leaves $B_\pm$ are both of codimension two and that the bundles $\sph^1 \to L \to B_\pm$ are both orientable.  If $b_2(M^7) \leq 2$, then $M^7$ is rationally elliptic.
\end{thm}


\begin{proof}
By Table \ref{table:Qtype}, the hypotheses are equivalent to letting $\pi_1(F) = \Z^2$.  Therefore, $F$ is rationally homotopy equivalent to $\sph^1 \x \sph^1 \x \Omega \sph^3$.

The long exact homotopy sequence for the homotopy fibration $F \to L \to M^7$ yields
\beq
\label{E:FLM}
\dots \to \pi_3(M^7) \to \pi_2(F) \to \pi_2(L) \to \pi_2(M^7) \to \pi_1(F) \to \pi_1(L) \to 0\,,
\eeq
from which it is clear that $\rank(\pi_1(L)) \leq 2$, with equality if and only if $\pi_1(L) = \Z^2$.

Assume first that $\pi_1(L) = \Z^2$.  
Theorem \ref{T:NoRank} and \eqref{E:FLM} 
together imply that $\rk(\pi_2(L)) = \rk(\pi_2(M^7))$ and hence, by the Hurewicz Theorem, that $\rk(\pi_2(L)) = b_2(M^7) \leq 2$.

By Proposition \ref{P:PDmiddle}, the rational cohomology ring of the maximal free abelian cover $\ol L$ of $L$ is isomorphic to that of a closed, simply connected, four-dimensional manifold $N^4$.  Given
$$
b_2(N^4) = b_2(\ol L) = \rk(\pi_2(\ol L)) = \rk(\pi_2(L)) \leq 2\,,
$$
it follows from Freedman's classification of smooth, closed, simply connected $4$-manifolds \cite{Fr} that $N^4$ is homeomorphic to one of $\sph^4$, $\CP^2$, $\sph^2 \x \sph^2$ or $\CP^2 \# \pm \CP^2$ and, hence, rationally elliptic.  Moreover, in \cite{Mill} (see also \cite{Lu}) Miller proved that, for all $k \geq 2$, if $X$ is a (rationally) $(k-1)$-connected space of formal dimension $\leq 4k-2$ with $H^*(X; \Q)$ satisfying Poincar\'e duality, then $X$ is intrinsically formal.  In the present setting, this implies that $\ol L$ is intrinsically formal and, hence, that its minimal model is isomorphic to that of the rationally elliptic space $N^4$.  Therefore, as $\pi_j(\ol L) = \pi_j(L)$ for all $j \geq 2$, the rational ellipticity of $M^7$ now follows from Lemma \ref{L:high_enough}.

Assume now that $\rank(\pi_1(L)) \leq 1$.  By Lemma \ref{L:pi1B}, at least one of $\pi_1(B_\pm)$ is finite.  Let $B \in \{B_\pm\}$ such that $\pi_1(B)$ is finite and let $\tilde B$ be its universal cover, a closed, smooth, simply connected $5$-manifold.  From the classification of Barden and Smale \cite{Ba,Sm3}, together with the Hurewicz Theorem, it follows that $\tilde B$, and therefore $B$, is rationally elliptic if $\rank(\pi_2(B)) = \rank(\pi_2(\tilde B)) = b_2(\tilde B) \leq 1$.  In this case, the rational ellipticity of $M^7$ follows immediately from Lemma \ref{L:high_enough}.

To establish that $\rank(\pi_2(B)) \leq 1$, observe that exactness in the long exact homotopy sequence for the bundle $\sph^1 \to L \to B$ yields
$$
\rank(\pi_2(B)) = 1 + \rank(\pi_2(L)) - \rank(\pi_1(L)) \,.
$$
On the other hand, exactness in \eqref{E:FLM}, together with Theorem \ref{T:NoRank} and the Hurewicz Theorem, yields
$$
 \rank(\pi_2(L)) - \rank(\pi_1(L)) = \rk(\pi_2(M^7)) - \rk(\pi_1(F)) = b_2(M^7) - 2\,.
$$
Since $b_2(M^7) \leq 2$ by hypothesis, these identities yield $\rank(\pi_2(B)) = b_2(M^7) - 1 \leq 1$, as desired.
\end{proof}

In the remaining case with singular leaves of codimension two to be discussed below, where exactly one of the singular leaves is orientable, it turns out there are no such double disk bundles whenever $b_2(M^7) \geq 3$.  An example of such a decomposition can be found on $\sph^3 \x \CP^2$ by taking advantage of the fact that $\CP^2$ decomposes as the union of disk bundles over $\sph^2$ and $\RP^2$ \cite{GeRa}.  


\begin{thm}
\label{P:rk1}	
Suppose that the singular leaves $B_\pm$ are both of codimension two and that exactly one of the bundles $\sph^1 \to L \to B_\pm$ is orientable.  Then $M^7$ is rationally elliptic.
\end{thm}


\begin{proof}This proof will follow the same basic strategy used in the proof of Theorem \ref{T:nonoriented}, but the computation is significantly more involved.

By Table \ref{table:Qtype}, the hypotheses are equivalent to letting $\pi_1(F) = \Z \oplus \Z_2$ and, hence, that $F$ is rationally homotopy equivalent to $\sph^1 \x \sph^3 \x \Omega \sph^5$.  Observe now, using the long exact homotopy sequences for $F \to L \to M^7$ and $\sph^1 \to L \to B_\pm$, that $\pi_1(L)$ and $\pi_1(B_\pm)$ are abelian groups satisfying $\rk(\pi_1(B_\pm)) \leq \rk(\pi_1(L)) \leq \rk(\pi_1(F))= 1$. In particular, $H_1(L) \cong \pi_1(L)$ and $H_1(B_\pm) \cong \pi_1(B_\pm)$.

Without loss of generality, suppose that $\sph^1 \to L \to B_-$ is orientable and $\sph^1 \to L \to B_+$ is non-orientable.  By \cite[p.~107]{Hi}, $L$ is orientable and, hence, $B_-$ is a closed, orientable $5$-manifold, while $B_+$ is non-orientable.  In particular, $L$ and $B_-$ satisfy Poincar\'e duality, whereas $H_5(B_+) = 0$.

Consider the pairs $(M^7, DB_\pm)$ and $(DB_\pm, L)$.  By excision and Poincar\'e-Lefschetz duality, and recalling that $DB_\pm$ is homotopy equivalent to $B_\pm$, there are isomorphisms
\beq
\label{E:PLD}
H^j(M, DB_\pm) \cong H^j(DB_\mp, L) \cong H_{7-j}(B_\mp) \,,
\eeq
for all $j \geq 0$.  Thus, the portion
$$
\dots \to H^1(M) \to H^1 (B_-) \to H^2(M, DB_-) \to \cdots
$$
of the long exact sequence for the pair $(M^7, DB_-)$ yields that $H^1(B_-) = 0$.  Applying the Universal Coefficient Theorem, it may be deduced that $\pi_1(B_-) \cong H_1(B_-)$ is finite.  

As a result, the universal cover $\tilde B_-$ of $B_-$ is a closed, simply connected $5$-manifold.  By the classification of Barden and Smale \cite{Ba,Sm3}, together with the Hurewicz Theorem, it follows that $\tilde B_-$, and therefore $B_-$, is rationally elliptic provided that $\rank(\pi_2(B_-)) = \rank(\pi_2(\tilde B_-)) = b_2(\tilde B_-) \leq 1$.  In this case, the rational ellipticity of $M^7$ follows immediately from Lemma \ref{L:high_enough}.

To show that $M^7$ is rationally elliptic, it therefore suffices to show that $\rk(\pi_2(B_-)) \leq 1$.  To this end, assume instead that $\rk(\pi_2(B_-)) \geq 2$.  It will be demonstrated, by placing the focus on $b_3(M^7)$, that this assumption leads to a contradiction.  Some initial setup is required.

Observe that the long exact sequence for the pair $(DB_+, L)$ yields
$$
\dots\to H^5(B_+) \to H^5 (L) \to H^6(DB_+, L) \to H^6(B_+) = 0 \,.
$$
By the Universal Coefficient Theorem, $H^5(B_+)$ must be finite.  Therefore, it follows from Poincar\'e duality and \eqref{E:PLD} that 
\beq
\label{E:pi1LB}
\rk(\pi_1(L)) = b_1(L) = b_5(L) = b_1(B_+) = \rk(\pi_1(B_+))\,.
\eeq

Since $\sph^1 \to L \to B_-$ is orientable, there exists a Gysin sequence
$$
\dots \to H^{j+1}(L) \to H^j(B_-) \to H^{j+2}(B_-) \to H^{j+2}(L) \to H^{j+1}(B_-) \to \cdots
$$
and, because $b_4(B_-) = b_1(B_-) = 0$ and $b_2(B_-) = b_3(B_-)$ by Poincar\'e duality, it follows that
\beq
\label{E:b3L}
b_3(L) = 2 \, b_2(B_-) = 2 (b_2(L) - b_1(L) + 1)\,.
\eeq 

Since $F \simeq_\Q \sph^1 \x \sph^3 \x \Omega \sph^5$, one obtains from the long exact homotopy sequences for $F \to L \to M^7$ and $\sph^1 \to L \to B_\pm$ that
\beq
\label{E:b2M}
\begin{split}
b_2(M^7) = \rk(\pi_2(M^7)) &= \rk(\pi_2(L)) - \rk(\pi_1(L)) + 1 \\
&= \rk(\pi_2(B_\pm)) - \rk(\pi_1(B_\pm))
\end{split}
\eeq
and, hence, that $b_2(M^7) = \rk(\pi_2(B_-)) \geq 2$.

Consider now the rational Serre spectral sequence $(E_j, d_j)$ associated to the homotopy fibration $F \to L \to M^7$, where the rational cohomology of $F$ is given by \eqref{E:FQ8} and the $E_2$-page is shown in Figure \ref{F:specseq}.   Recall that $E_2^{k,l} = E_2^{k,0} \ox E_2^{0,l}$ for all $k, l \geq 0$.  As in the proof of Theorem \ref{T:nonoriented}, it is convenient to denote by $d_j^{k,l}$ the differential $d_j : E_j^{k,l} \to E_j^{k+j, l+1-j}$ and by $\Delta_j^m$ the diagonal $\{E_j^{k,l} \mid k + l = m\}$ on the $E_j$-page of $(E_j,d_j)$.
\begin{figure}
\centering
\begin{tikzpicture}
  \matrix (m) [ampersand replacement=\&, matrix of math nodes,
    nodes in empty cells,nodes={minimum width=5ex,
    minimum height=5ex,outer sep=-5pt},
    column sep=1ex,row sep=1ex]{
                   \&       \&    \&    \&   \&   \&   \&   \&   \&    \\
                F  \&   4    \&  \Q^2   \&  0   \&   \&   \&   \&   \&   \&    \\
                \&   3    \&  \Q   \&  0   \& \Q^{b_2(M^7)}  \& \Q^{b_3(M^7)}  \& \Q^{b_3(M^7)}  \& \Q^{b_2(M^7)}  \&  0 \&    \\
                \&   2   \&  0   \&   0  \& 0  \&  0 \&  0 \&  0 \& 0  \&    \\
                \&   1   \&  \Q   \&   0  \& \Q^{b_2(M^7)}  \& \Q^{b_3(M^7)}  \& \Q^{b_3(M^7)}  \& \Q^{b_2(M^7)}  \&  0 \&    \\
                \&   0   \&  \Q   \&  0   \& \Q^{b_2(M^7)}  \& \Q^{b_3(M^7)}  \& \Q^{b_3(M^7)}  \& \Q^{b_2(M^7)}  \&  0  \&    \\
                \&   \quad\strut   \&  0    \&  1   \&  2  \& 3  \&  4 \&  5 \&  6  \&  \strut \quad M^7 \\
                 };
\draw[thick] (m-1-2.east) -- (m-7-2.east) ;
\draw[thick] (m-7-2.north) -- (m-7-10.north) ;
\end{tikzpicture}
\caption{$E_2$-page of spectral sequence for $F \to L \to M^7$}
\label{F:specseq}
\end{figure}

From the differential $d_2^{0,1} = (d_2 : E_2^{0,1} = \Q \to E_2^{2,0} = \Q^{b_2(M^7)})$ and \eqref{E:b2M} it is clear that 
\beq
\label{E:b2L}
b_2(L) = b_2(M^7) - (1 - b_1(L)) = b_2(M^7) - 1 + \rk(\pi_1(L)) = \rk(\pi_2(L))
\eeq
and, therefore, by combining equations \eqref{E:b3L} and \eqref{E:b2M}, that
\beq
\label{E:b3LM}
b_3(L) = 2\, b_2(M^7) \geq 4\,.
\eeq

Suppose first that $b_3(M^7) \leq 1$.  If $b_3(M^7) = 0$, then, for all $j \geq 2$, all differentials involving terms along the diagonal $\Delta_j^3$ are trivial and, hence, $H^3(L; \Q) = \bigoplus_{k + l = 3} E_\infty^{k,l} = \Q^{b_2(M^7) + 1}$.  However, since $b_2(M^7) \geq 2$, this implies that $b_3(L) = b_2(M^7) + 1 < 2 \, b_2(M^7)$, contradicting the inequality \eqref{E:b3LM}.

If, on the other hand, $b_3(M^7) = 1$, then, by \eqref{E:b3LM},
$$
4 \leq 2 \, b_2(M^7) = b_3(L) \leq \rk\left( \bigoplus_{k + l = 3} E_2^{k,l}  \right) = b_2(M^7) + 2 \,,
$$
from which it follows that $b_2(M^7) = 2$ and $b_3(L) = 4$.  Moreover, this implies that, for all $j \geq 2$, all differentials involving terms along the diagonal $\Delta_j^3$ are trivial.  As a result, $H^4(L; \Q)$ is entirely determined by the kernels of differentials with domain along the diagonal $\Delta_j^4$.  The product rule implies that $\rk(d_2^{3,1}) \leq \rk(d_2^{0,1}) = 1 - b_1(L)$, while the total rank $r$ of all differentials $d_j^{0,4}$, $j \geq 2$, is clearly at most $2$.  Therefore, Poincar\'e duality and \eqref{E:b2L} together yield 
$$
1 + b_1(L) = b_2(L) = b_4(L) \geq 4 - \rk(d_2^{3,1}) - r \geq 1 + b_1(L)
$$
and, hence, the identities $\rk(d_2^{3,1}) = 1 - b_1(L)$ and $ r = 2$.

Given that $d_j^{0,3}$ is trivial for all $j \geq 2$, it follows from the product rule that $d_j^{2,3}$ is also trivial for all $j \geq 2$.  Therefore, the differentials with domain in $W_j = E_j^{2,3} \oplus E_j^{4,1} \oplus E_j^{0,5}$ are trivial for all $j \geq 2$, meaning that the contribution of $W_2 = \Q^5$ to $H^5(L; \Q)$ has rank $5 -  \rk(d_2^{3,1}) - r = 2 + b_1(L)$.  However, since $b_5(L) = b_1(L) \leq 1$ by Poincar\'e duality, this is impossible.

Suppose, finally, that $b_3(M^7) \geq 2$.  From the ring structure of $H^*(F; \Q) = H^*(\sph^1 \x \sph^3 \x \Omega \sph^5; \Q)$, it is clear that $\rk(d_2^{0,1}) = 1 - b_1(L)$ and $d_2^{0,3} = 0$ together imply that $\rk(d_2^{0,4}) \geq 1 - b_1(L)$.  Thus, $E_4^{0,4} = E_3^{0,4}$ has rank $\leq 1 + b_1(L)$ and, consequently, the image of $d_4^{0,4}$ has rank $\leq 1 + b_1(L)$ in $E_4^{4,1}$.  As $d_4^{0,4}$ is the only possible non-trivial differential involving $E_j^{4,1}$, $j \geq 2$, it follows from Poincar\'e duality that
$$
b_1(L) = b_5(L) \geq \rk(E_\infty^{4,1}) \geq b_3(M^7) - (1 + b_1(L))
$$
and may, therefore, be deduced that $2 \leq b_3(M^7) \leq 2 \, b_1(L) + 1$.  This forces $b_1(L) = 1$, which, in turn implies that $d_2^{k,1}$ is trivial for all $k \geq 0$.

Since $d_2^{2,1}$ is trivial, it follows that 
$$
2 \, b_2(M^7) = b_3(L) \geq \rk(E_\infty^{2,1} \oplus E_\infty^{3,0}) = b_2(M^7) + b_3(M^7)
$$
and, therefore, that $b_2(M^7) \geq b_3(M^7) \geq 2$.  Now, from this inequality and the fact that $d_2^{3,1}$ and $d_2^{4,1}$ are trivial, one obtains
\begin{align*}
1 = b_1(L) = b_5(L) &\geq \rk(E_2^{2,3} \oplus E_2^{4,1} \oplus E_2^{5,0}) - \rk(d_4^{0,4}) - \rk(d_5^{0,4}) \\
&\geq 2 \, b_2(M^7) + b_3(M^7) - 2 \\
&\geq 4\,,
\end{align*}
which is absurd.  This completes the proof.
\end{proof}

It remains only to deal with the cases where there is at least one singular leaf of codimension three.  Recall first that, for all $p,q \geq 0$, the sphere $\sph^{p+q+1}$ can be decomposed as $\sph^{p+q+1} = (\sph^p \x D^{q+1}) \cup (D^{p+1} \x \sph^q)$.  In particular, this implies that $\sph^7$ and, by Proposition \ref{P:SuffCond}, every $\sph^3$-bundle over $\sph^4$ admits a double disk-bundle decomposition with $\{\ell_\pm\} = \{1,2\}$.  From a rational homotopy perspective, this is all that can happen.


\begin{thm}
\label{P:codim23}
	If the bundles $\sph^{\ell_\pm} \to L \to B_\pm$ have $\{\ell_\pm\} = \{1,2\}$,  
then $M^7$ is rationally homotopy equivalent to $\sph^7$ or $\sph^3 \x \sph^4$.
\end{thm}


\begin{proof}  
Suppose, without loss of generality, that $\ell_- = 1$, $\ell_+ = 2$.  By Table \ref{table:Qtype}, both of the bundles $\sph^{\ell_\pm} \to L \to B_\pm$ are orientable, $\pi_1(F) = \Z$ and 
$F$ is rationally homotopy equivalent to $\sph^1 \x \sph^2 \x \Omega \sph^4$.  Moreover, since $M^7$ is simply connected and $\ell_+ = 2$, equation (3.7) of \cite{GH} implies that the homomorphism $\pi_1(\sph^1) \to \pi_1(L)$ in the long exact homotopy sequence for $\sph^1 \to L \to B_-$ must be surjective.  In particular, it follows that $\pi_1(B_-) = 0$.  On the other hand, the long exact homotopy sequence for $\sph^2 \to L \to B_+$ yields $\pi_1(L) = \pi_1(B_+)$ and, thus, $b_1(L) = b_1(B_+)$, while the long exact homotopy sequence for $F \to L \to M^7$ yields that either $\pi_1(L) = \Z$ or $\pi_1(L)$ is finite.  Observe, finally, that the orientability of the bundles $\sph^{\ell_\pm} \to L \to B_\pm$ and \cite[p.~107]{Hi} together ensure that $L$ and $B_\pm$ are all orientable and, hence, satisfy Poincar\'e duality.

In order to establish that $M^7$ is rationally elliptic, it suffices, via the classification of Barden and Smale \cite{Ba,Sm3}, the Hurewicz Theorem and Lemma \ref{L:high_enough}, to demonstrate that $b_2(B_-) \leq 1$, because $B_-$ is a closed, simply connected $5$-manifold. 

From exactness in the portion of the Gysin sequence for the bundle $\sph^1 \to L \to B_-$ given by
$$
0 = H^1(B_-) \to H^1 (L) \to H^0 (B_-) \to \cdots 
\to H^3(L) \to H^2(B_-) \to H^4(B_-) = 0 \,,
$$
together with Poincar\'e duality, it may easily be deduced that
\beq
\label{E:b2b3L}
b_3(L) = 2 \, b_2(B_-) = 2  (b_2(L) - b_1(L) + 1) .
\eeq
On the other hand, exactness in the portion of the Gysin sequence for $\sph^2 \to L \to B_+$ given by 
$$
0 = H^{-1}(B_+) \to H^2(B_+) \to H^2(L) \to H^0(B_+) \to  \cdots 
\to H^2(B_+) \to H^5(B_+) = 0
$$
yields
\beq
\label{E:b2B}
2 \, b_2(B_+) = 2(b_2(L) + b_1(L) - 1) - b_3(L) \,.
\eeq

By combining equations \eqref{E:b2b3L} and \eqref{E:b2B}, it may be concluded that $b_2(B_+) = 2(b_1(L) - 1) \leq 0$, whence it follows that $b_1(B_+) = b_1(L) = 1$ and $b_2(B_+) = 0$.  From the Gysin sequence for $\sph^2 \to L \to B_+$ it now follows that $b_2(B_-) = b_2(L) \leq b_0(B_+) = 1$ and, therefore, $M^7$ is rationally elliptic.

Furthermore, in the rational long exact sequence for the pair $(M^7, DB_+)$ there is a short exact sequence
$$
0 = H^1(M^7; \Q) \to H^1(B_+; \Q) \to H^2(M^7, DB_+; \Q) \to H^2(M^7; \Q) \to H^2(B_+; \Q) = 0 \,.
$$
Since, by excision and Poincar\'e-Lefschetz duality, 
$$ 
H^2(M^7, DB_+; \Q) \cong  H^2(DB_-, L; \Q) \cong  H^5(B_-; \Q) = \Q,
$$
it follows from the Hurewicz Theorem that $\rk(\pi_2(M^7)) = b_2(M^7) = 0$.  Now, by \cite{He}, or by simply examining the inequalities in \eqref{E:degrels}, a rationally elliptic manifold $M^7$ with $\rk(\pi_2(M^7)) = 0$ must be rationally homotopy equivalent to $\sph^7$ or $\sph^3 \x \sph^4$, as desired.
\end{proof}

By taking advantage of the standard decomposition of $\sph^3$ as the union of two $3$-disks, it is clear that, for any closed, smooth, simply connected $4$-manifold $N^4$, the product $\sph^3 \x N^4$ admits a double disk-bundle decomposition with singular leaves both of codimension three.  Therefore, to avoid rationally hyperbolic $7$-manifolds admitting such a double disk-bundle decomposition, it is necessary to impose some topological restrictions.


\begin{thm}
\label{P:codim3}
	If the singular leaves $B_\pm$ are both of codimension three and $b_2(M^7) \leq 2$, then $M^7$ is rationally elliptic.	
\end{thm}

\begin{proof}
By Table \ref{table:Qtype}, the hypothesis on the singular leaves is equivalent to $F$ being simply connected and rationally homotopy equivalent to one of
$$
\sph^2 \x \Omega \sph^3, \ 
\sph^2 \x \sph^2 \x \Omega \sph^5, \ 
\SU(3)/T^2 \x \Omega \sph^7, \ 
\Sp(2)/T^2 \x \Omega \sph^9,  \text{ or }
\mathrm{G}_2/T^2 \x \Omega \sph^{13}.
$$
In particular, observe that in all cases $\rk(\pi_2(F)) = 2$ and $1 \leq \rk(\pi_3(F)) \leq 2$.

From the long exact homotopy sequences for $F \to L \to M^7$ and $\sph^2 \to L \to B_\pm$, it is clear that $\pi_1(L) = \pi_1(B_\pm) = 0$.  In particular, each of the bundles $\sph^2 \to L \to B_\pm$ possesses a Gysin sequence, from which it may easily be concluded that $H^2(L) = H^2(B_\pm) \oplus \Z$ and $H^3(L) = 0$, given that $B_\pm$ are closed, simply connected $4$-manifolds.

By the classification of closed, simply connected $4$-manifolds \cite{Fr} and Lemma \ref{L:high_enough}, it suffices to show that $b_2(B_\pm) \leq 2$ in order to establish that $M^7$ is rationally elliptic.  Suppose to the contrary, therefore, that $3 \leq b_2(B_\pm) = b_2(L) - 1$. 

The hypothesis $b_2(M^7) \leq 2$, together with exactness in the portion
$$
0 = H^1(L) \to H^2(M^7) \to H^2(B_-) \oplus H^2(B_+) \to H^2(L) \to H^3(M^7) \to H^3(B_-) \oplus H^3(B_+) = 0
$$
of the Mayer-Vietoris sequence for $M^7 = DB_- \cup_L DB_+$, now yields 
\begin{align*}
2 \geq b_2(M^7) &= b_3(M^7) + b_2(B_-) + b_2(B_+) - b_2(L) \\
&= b_3(M^7) + b_2(B_-) - 1 \\
&\geq b_3(M^7) + 2 \\
&\geq 2\,,
\end{align*}
from which it immediately follows that $b_2(M^7) = 2$, $b_3(M^7) = 0$, $b_2(B_\pm) = 3$ and $b_2(L) = 4$.  In particular, by the Hurewicz Theorem, it is now clear that $\rk(\pi_2(M^7)) = 2$ and $\rk(\pi_2(L)) = 4$.

Let $(\wedge V_X, d_X)$ be the minimal model of a simply connected space $X$, where $V_X = \bigoplus_{j=0}^\infty V^j$, with $V^0 = \Q$, $V^1 = 0$ and $V^j \cong \pi_j^\Q(X)$, for all $j \geq 2$.  Recall that $d_X$ is decomposable and, hence, satisfies $d_X(V^2) = 0$ and $d_X(V^3) \In V^2 \cdot V^2 \In \ker(d_X)$.  In particular, if $H^3(X; \Q) = 0$, then $d_X$ must map $V^3$ injectively into $V^2 \cdot V^2$.  Therefore, 
$$
\rk(\pi_3(X)) = \dim_\Q(V^3) \leq \dim_\Q(V^2 \cdot V^2) = \frac{1}{2} \rk(\pi_2(X)) (\rk(\pi_2(X)) + 1),
$$
while
\begin{align*}
b_4(X) &\geq \dim_\Q(V^2 \cdot V^2) - \dim_\Q(V^3) \\
&= \frac{1}{2} \rk(\pi_2(X)) (\rk(\pi_2(X)) + 1) - \rk(\pi_3(X)) \geq 0\,.
\end{align*}

Now, by Poincar\'e duality, $b_4(M^7) = b_3(M^7) = 0$ and $b_4(L) = b_2(L) = 4$.  Thus, the inequalities above, together with the identities $\rk(\pi_2(M^7)) = 2$ and $\rk(\pi_2(L)) = 4$, yield
$$
\rk(\pi_3(M^7)) = 3 
\quad \text{ and } \quad 
\rk(\pi_3(L)) \geq 10 - b_4(L) = 6\,.
$$
However, from the long exact homotopy sequence for $F \to L \to M$ one has
$$
\rk(\pi_3(L)) \leq \rk(\pi_3(M^7)) + \rk(\pi_3(F)) \leq 5 \,,
$$
a contradiction.
\end{proof}

The main result of this section is now a simple consequence of all the preceding groundwork.


\begin{thm}
\label{T:dim7}
Let $M^7$ be a closed, smooth, simply connected $7$-manifold with second Betti number $b_2(M^7) \leq 2$ which admits a double disk-bundle decomposition.  Then $M^7$ is rationally elliptic.
\end{thm}


\begin{proof}
By Proposition \ref{P:L_BOUND}, it suffices to consider singular leaves of codimension at most three.  All such cases has been dealt with in Theorems~\ref{P:rk0}, \ref{P:rk2}, \ref{P:rk1}, \ref{P:codim23}, and \ref{P:codim3}, thus completing the proof.
\end{proof}


\begin{remark}
\label{r:7hyp}
Notice that the hypothesis $b_2(M^7) \leq 2$ has been used only in the cases where the singular leaves $B_\pm$ are of codimension $1 \leq \ell_- = \ell_+ \leq 2$ and the bundles $\sph^{\ell_\pm} \to L \to B_\pm$ are both orientable.  In these scenarios, standard decompositions of $\sph^3$ lead to counterexamples whenever $b_2(M^7) \geq 3$ is permitted.  In all other cases, assuming only that $M^7$ admits a double disk-bundle decomposition is enough to conclude that $M^7$ is rationally elliptic.
\end{remark}

In contrast with the six-dimensional case, recall that a large family of closed, $2$-connected $7$-manifolds admitting double-disk bundle decompositions was constructed in \cite{GKS1}, each having rational cohomology ring isomorphic to either $H^*(\sph^7; \Q)$ or $H^*(\sph^3 \x \sph^4; \Q)$.  Moreover, in \cite{GKS2} it was observed that this family does not contain all possible homotopy types of such manifolds: for example, it does not contain any $2$-connected $7$-manifold $M^7$ with $H^4(M^7) = \Z_5$ and non-standard linking form.  It is unknown whether these excluded spaces also admit a double disk-bundle decomposition.

Being unable to address even the case of rational $7$-spheres at present, a classification up to diffeomorphism of simply connected, rationally elliptic $7$-manifolds which admit a double disk-bundle decomposition seems out of reach for the moment.  Recall, however, that Herrmann has shown that a simply connected, rationally elliptic $7$-manifold must be rationally homotopy equivalent to one of $\sph^7$, $\sph^2 \x \sph^5$, $\sph^3 \x \sph^4$, $\sph^3 \x \CP^2$, $N^7$ or $M^7_\sigma$, for $\sigma \in \Q^*/(\Q^*)^2$.  The manifold $N^7$ has minimal model $(\wedge V, d) = (\wedge(x_1, x_2, y_1, y_2, y_3), d)$, where $\deg(x_i) = 2$, $i \in \{1,2\}$, $\deg(y_j) = 3$, $j \in \{1,2,3\}$, and the differential $d$ is given by
\beq
\label{E:N7}
d(x_1) = d(x_2) = 0\,, \qquad 
d(y_1) = x_1^2 \,, \qquad 
d(y_2) = x_2^2 \,, \qquad 
d(y_3) = x_1 x_2 \,.
\eeq
By \cite[Theorem 6.1]{GGKR}, the unit tangent bundle of $\sph^2 \x \sph^2$ is a concrete representative of this rational homotopy type.  Moreover, from Example 2.91 of \cite{FOT} it is known that any manifold with minimal model \eqref{E:N7} is not formal.

The family $M^7_\sigma$ consists of spaces with minimal model $(\wedge V_\sigma, d_\sigma) = (\wedge(x_1, x_2, y_1, y_2, y_3), d_\sigma)$, where $\deg(x_i) = 2$, $i \in \{1,2\}$, $\deg(y_j) = 3$, $j \in \{1,2,3\}$, and the differential $d_\sigma$ is given, for $\sigma \in  \Q^*/(\Q^*)^2$, by
\beq
\label{E:Msig}
d_\sigma(x_1) = d_\sigma(x_2) = d_\sigma(y_3) = 0 \,, \qquad 
d_\sigma(y_1) = x_1 x_2 \,, \qquad 
d_\sigma(y_2) = x_1^2 - \sigma x_2^2 \,.
\eeq
The spaces $M^7_1$ and $M^7_{-1}$ are rationally homotopy equivalent to $\sph^3 \x (\CP^2 \# \CP^2)$ and $\sph^3 \x (\CP^2 \# \ol{\CP}^2) \simeq_\Q \sph^3 \x \sph^2 \x \sph^2$ respectively, whereas no concrete representative of the rational homotopy type is currently known when $\sigma \neq \pm 1$.

\begin{thm}
\label{T:MNDDBs}
Each of the minimal models \eqref{E:N7} and \eqref{E:Msig}, $\sigma = \pm 1$, is realized by biquotients $( \sph^3 \x \sph^3 \x \sph^3) \bq T^2$ and, hence, has a representative admitting a double disk-bundle decomposition.  Moreover, there are infinitely many such biquotients with minimal model \eqref{E:N7}, each of which is not formal.
\end{thm}

\begin{proof}
By \cite[Example 2.91]{FOT}, a space with minimal model \eqref{E:N7} is not formal.  Biquotients of the form $( \sph^3 \x \sph^3 \x \sph^3) \bq T^2$ have been studied in \cite{DV2} and \cite{GGKR}.  The minimal models were determined in the proof of \cite[Theorem 6.1]{GGKR}, while the integral cohomology rings and characteristic classes were determined in \cite[Proposition 4.35]{DV2}.  In particular, a biquotient with minimal model \eqref{E:N7} generically has torsion in its cohomology ring: for example, there is a nice subfamily $N^7_m$, $m \in \Z$, of such spaces consisting of $\sph^3$-bundles over $\sph^2 \x \sph^2$ with structure group $T^2$, $H^2(N^7_m) = \Z^2$, $H^3(N^7_m) = 0$ and $H^4(N^7_m) = \Z_{m^2}$.  This subfamily is described by the action
\begin{align*}
T^2 \x ( \sph^3 \x \sph^3 \x \sph^3) &\to ( \sph^3 \x \sph^3 \x \sph^3) \\
((z,w), (q_1, q_2, q_3)) &\mapsto ( z \, q_1, w \, q_2 , z^m \, u_3 + w^m \, v_3 \, j) ,
\end{align*}
where $m \in \Z$ and $q_3 = u_3 + v_3 \, j \in \sph^3$, $u_3, v_3 \in \C$, $|u_3|^2 + |v_3|^2 = 1$.  By the proof of \cite[Theorem 6.1]{GGKR}, the unit tangent bundle of $\sph^2 \x \sph^2$ is given by setting $m = 2$.

The fact that all biquotients $( \sph^3 \x \sph^3 \x \sph^3) \bq T^2$ admit a double disk-bundle decomposition follows from Proposition \ref{P:SuffCond}\eqref{i:BiqFol}, since the free $T^2$ action on $\sph^3 \x \sph^3 \x \sph^3$ is a subaction of a cohomogeneity-one action by $T^2 \x (\sph^3 \x \sph^3) \x (\sph^3 \x \sph^3)$.
\end{proof}

Finally, note that it is unknown whether there is a representative of each rational homotopy type \eqref{E:Msig}, $\sigma \neq \pm 1$, which admits either a double disk-bundle decomposition or a metric with non-negative sectional curvature.



\end{document}